\documentclass[3p]{elsarticle}
\biboptions{sort&compress}

\makeatletter
\def\ps@pprintTitle{%
  \let\@oddhead\@empty
  \let\@evenhead\@empty
  \let\@oddfoot\@empty
  \let\@evenfoot\@empty}
\makeatother

\usepackage[bookmarks=true,colorlinks=true,linkcolor=blue]{hyperref}
\usepackage{amsmath,amsfonts,amssymb,mathabx,bm}
\usepackage{graphicx,esint,ulem}
\usepackage{calc}
\usepackage{geometry}
\usepackage[usenames,dvipsnames,svgnames,table]{xcolor}
\usepackage{stmaryrd}   
\usepackage{fancyvrb} 

\usepackage{color}

\usepackage{algorithm,algpseudocode}
\usepackage{algorithmicx}
\algrenewcommand\alglinenumber[1]{\footnotesize #1:} 

\usepackage{tikz}
\usepackage{caption}

\usepackage{tcolorbox}
\usepackage[english]{babel}
\usepackage{fullpage,setspace,float}
\usepackage{tabu}
\usepackage{multirow}
\usepackage{mdframed}
\usepackage{cellspace} 
\usepackage{hhline}     
\newcommand{\ul}{u_l}

\newcommand{\uh}{u_h}

\newcommand{\calNd}{\mathcal N_{\delta}}

\newcommand{\Mtyp}{M_{\text{typ}}}
\newcommand{\sighat}{S}                   
\newcommand{\sinc}{\text{sinc}}

\DeclareMathOperator{\diam}{diam}

\newcommand{\pt}{\partial_t}

\newcommand{\dt}{\Delta t}

\newcommand{\dk}{\Delta k}

\newcommand{\beq}{\begin{equation}}
\newcommand{\eeq}{\end{equation}}
\newcommand{\beqs}{\begin{equation*}}
\newcommand{\eeqs}{\end{equation*}}

\newcommand{\Nx}{N_x}

\newcommand{\Nt}{N_t}
\newcommand{\bigO}{{\mathcal O}}

\newcommand{\xv}{\bold x}

\newcommand{\yv}{\bold y}

\newcommand{\abs}[1]{\left\vert #1\right\vert}

\newcommand{\kv}{\bold k}

\newcommand{\calE}{\mathcal{E}}

\newcommand{\norm}[1]{\Vert #1\Vert}

\newcommand{\calG}{\mathcal{G}}
\newcommand{\CGA}{\calG_A}

\newcommand{\nv}{\bold n}
\newcommand{\mv}{\bold m}
\newcommand{\erf}{\text{erf}}

\newcommand{\Oh}[1]{\mathcal O\left(#1\right)}

\newcommand{\R}{\mathbb{R}}          
\newcommand{\Z}{\mathbb{Z}}          

\newtheorem{theorem}{Theorem}

\newenvironment{proof}[1][Proof]{\begin{trivlist}
\item[\hskip \labelsep {\bfseries #1.}]}{\end{trivlist}}

\newtheorem{prop}{Proposition}[section]
\newtheorem{lem}{Lemma}[section]

\newtheorem{remark}{Remark}[section]

\begin{document}

\def\ccm{Center for Computational Mathematics, Flatiron Institute, Simons Foundation,
  New York, New York 10010}

\def\nyu{Courant Institute of Mathematical Sciences,
  New York University, New York, New York 10012}

\begin{frontmatter}
  \title{Truncated kernel windowed Fourier projection: a fast algorithm for the 3D free-space wave equation}
      
 \author{Nour G. Al Hassanieh\fnref{ccm}}
\address[ccm]{\ccm}
 \ead{nalhassanieh@flatironinstitute.org}

\author{Alex H. Barnett\fnref{ccm}}
\ead{abarnett@flatironinstitute.org}

\author{Leslie Greengard\fnref{ccm,nyu}}
\address[nyu]{\nyu}
\ead{lgreengard@flatironinstitute.org}

\begin{abstract}
We present a spectrally accurate fast algorithm for evaluating the solution 
to the scalar wave equation in free space driven by a large collection of point sources in a bounded domain. 
With $M$ sources temporally discretized by $\Nt$ time steps of size $\dt$,
a naive potential evaluation at $M$ targets on the same time grid
requires $\bigO(M^2 \Nt)$ work.
Our scheme requires $\Oh{(M + N^3\log N)\Nt}$ work, where
$N$ scales as $\Oh{1/\dt}$, i.e., the maximum signal frequency.
This is achieved by using the recently-proposed windowed Fourier projection (WFP) method to split
the potential into a local part, evaluated directly, plus a smooth history part
approximated by an $N^3$-point equispaced discretization of the Fourier transform, where
each Fourier coefficient obeys a simple recursion relation.
The growing oscillations in the spectral representation (which would be present with a naive
use of the Fourier transform) are controlled by spatially
truncating the hyperbolic Green's function itself. Thus, the method avoids the need for 
absorbing boundary conditions.
We demonstrate the performance of our algorithm with up to a million sources and targets at 6-digit accuracy.
We believe it can serve as a key component in addressing time-domain wave equation scattering problems.
\end{abstract}

\begin{keyword}
wave equation, free-space, hyperbolic potentials, high-order, truncated kernel, Fourier methods
\end{keyword}

\end{frontmatter}

\section{Introduction}

We present a fast algorithm to compute free-space hyperbolic potentials in 3D to high accuracy
without the need for artifical absorbing boundary conditions.
Such potentials are space-time convolutions of the free-space retarded
Green's function with a given source distribution; the latter
could be a singular density supported on a boundary (a layer potential)
\cite{Guenther2012,Costabel04,sayasbook},
a smooth density supported in a bounded domain (a volume potential),
or a set of point-like scatterers. 
For our present purposes, we assume
that the source distribution has already been spatially discretized
as a large collection of points.
Other aspects of potential theory, including
integral representations, quadrature design,
and the application of our fast algorithm to time-marching the integral equation solution
will be considered at a later date.
However, there are many advantages to using integral equation
methods \cite{CK83,Costabel04,Yilmaz2004,Banjai2014,Sayas,Barnett2020,Liu2020}. In particular, 
they do not require absorbing boundary conditions
to avoid spurious reflections
from a finite computational domain, and (in the homogeneous case),
they discretize the boundary alone, without the need for complicated mesh generation.
The rapid evaluation of such potentials accelerates time-domain integral equation-based
schemes for the solution of wave scattering problems.
The latter are essential in many scientific and engineering applications such as acoustics
\cite{kaltenbacher2018computational,wang2019}
and, generalizing to vectorial cases,
electromagnetic/optical scattering \cite{chewbook,CMS,CK83,Yilmaz2004,Liu2020}
and seismic propagation \cite{Virieux12}.
The time-domain approach to such problems is especially useful
when faced with wide-band signals and/or nonlinear materials.

In this paper, we consider 
the solution of the free-space wave equation
\begin{subequations}
\label{eq:freeSpace}
  \begin{align}
  \pt^2u(\xv,t) - \Delta u(\xv,t) &= f(\xv,t), \quad \xv\in\mathbb R^{3}, \; t \in (0,T],
    \\
u(\xv,0) = \pt u(\xv,0) &= 0, \qquad\quad\; \xv\in\mathbb R^{3}. 
\end{align}
\end{subequations}
The forcing function $f$ will be assumed to be a sum of $M$ point sources 
located at $\yv_j$, $j = 1, \dots, M$, with smooth time signatures $\sigma_j(t)$, 
such that $\sigma_j(t)=0$ for $t\leq 0$:
\beq\label{eq:forcing}
f(\xv,t) = \sum_{j = 1}^M\delta(\xv - \yv_j)\sigma_j(t),
\eeq
where $\delta(\xv)$ denotes the three-dimensional Dirac distribution.
We assume that the points $\yv_j$ are 
contained in a finite computational domain $B \equiv [-1,1]^3$ and that the 
solution is sought within $B$ as well.
In particular, we focus on the efficient evaluation of $u(\xv,t)$
at a large number of spatial targets $\xv\in B$, for a set of discrete time
slices $t\in(0,T]$.

The problem \eqref{eq:freeSpace} can be solved exactly using 
the free-space Green's function for the wave equation (the \textit{wave kernel})
\cite{Guenther2012},
given by
\beq\label{eq:Greens}
G(\xv,t) =
\frac{\delta(t - \abs{\xv})}{4\pi \abs{\xv}},
\eeq
where $\delta(t)$ denotes the one-dimensional Dirac distribution.
The exact solution to \eqref{eq:freeSpace} is the space-time convolution
\beq\label{eq:mainSolnRep}
u(\xv,t) = 
\int_0^t \int_{\mathbb{R}^3} G(\xv - \yv,t-\tau) f(\yv,\tau) 
d\yv d\tau.
\eeq
Given the fact that the forcing function takes the form
\eqref{eq:forcing}, then using the $\delta$-function form of the Green's function
(a manifestation of the strong Huygens' principle),
we may write
\begin{align}
u(\xv,t) &= \sum_{j=1}^M 
\int_0^t G(\xv - \yv_j,t-\tau) \sigma_j(\tau) d\tau 
\label{eq:mainSolnRep2a} \\
&=
\sum_{j=1}^M \frac{\sigma_j(t- \abs{\xv-\yv_j})}
    {4\pi \abs{\xv-\yv_j}},
    \qquad \xv \neq \yv_j, \, j=1,\dots,M.
\label{eq:mainSolnRep2b}
\end{align}
Direct evaluation of this retarded potential at $M$ targets
at a single time $t$ clearly has cost $\bigO(M^2)$.
(In an integral equation setting, such an evaluation would be needed
at each time step.)

Rather than making direct use of the free-space Green's function, as above, it is 
tempting to solve~\eqref{eq:freeSpace} using Fourier analysis.
For this, we define the spatial Fourier transform pair by 
\beq
\label{eq:fourierdef}
\hat u (\kv,t) = \int_{\mathbb R^3} u(\xv,t) e^{i\kv\cdot \xv}d\xv, \quad \text{and} \quad u (\xv,t) =\frac{1}{(2\pi)^3} \int_{\mathbb R^3} \hat u(\kv,t) e^{-i\kv\cdot \xv}d\kv.
\eeq
Fourier transformation of the scalar wave equation 
\eqref{eq:freeSpace} itself results in the ordinary differential equation (ODE)
\beq
\pt^2 \hat u(\kv,t) +\kappa^2 \hat u(\kv,t) = \sighat(\kv,t),
\eeq
wkere (here and throughout the paper) $\kappa := | \kv |$ denotes the wavevector magnitude.
The right-hand side driving $\sighat(\kv,t)$ is the spatial Fourier transform of the source $f$.
In our spatially discretized setting with the latter given by \eqref{eq:forcing}, we have
\beq\label{eq:sighat}
\sighat(\kv,t) = \sum_{j = 1}^M  \sigma_j(t) e^{i \kv \cdot \yv_j}.
\eeq
Solution of the ODE at each $\kv$ by Duhamel's principle
leads to the spectral representation 
\beq
\label{eq:uexactF}
u(\xv,t) = \frac{1}{(2\pi)^3} 
\int_{\mathbb{R}^3}\int_0^t  
e^{-i \kv \cdot \xv} \frac{\sin \kappa (t-\tau)}{\kappa} \sighat(\kv,\tau) d\tau
d\kv.
\eeq
By comparing \eqref{eq:uexactF} and \eqref{eq:mainSolnRep},
or by direct computation of $\hat G(\kv,t)$,
one finds that the spectral representation of the wave kernel is
\beq\label{eq:spectralGreens}
G(\xv, t) = \frac{1}{(2\pi)^3}\int_{\mathbb R^3} \frac{\sin \kappa t}{\kappa} e^{-i\kv\cdot\xv}d\kv,
\qquad t>0.
\eeq

At first glance, the Fourier-based solution \eqref{eq:uexactF}
is of little numerical interest for three simple reasons. First, the wavevector integrand 
is slowly decaying (note that the wave kernel decays like $1/\kappa$ while
$\sighat$ does not decay at all);\footnote{A more detailed analysis shows that, for $\sigma_j$ smooth,
after performing the time integral, the wavevector integrand decays as $1/\kappa^2$.
Truncation of this integral would thus still result in slow convergence.}
this is associated with the non-smooth ``switch on'' of the response at $\tau=t$.
Second, this integrand becomes more and more 
oscillatory as time advances, so that a finer and finer volumetric discretization of $\kv$ is needed
as $t \rightarrow \infty$. Third, the representation is still dependent on the entire
space-time history of $f$.
It turns out, however, that all three of these obstacles can be overcome, leading 
to our fast algorithm.

The first concern is addressed by the windowed Fourier projection (WFP) method,
introduced by the authors in \cite{wfp2025}, 
which ensures that the integrand in \eqref{eq:uexactF} is rapidly decaying with respect to $\kv$, so
long as one limits the Fourier representation to the ``history part'' of the solution.
More precisely, in the WFP approach, we
split the solution \eqref{eq:mainSolnRep2a} into a non-smooth local part, $\ul$, plus a smooth history 
part, $\uh$, using a temporal blending function, $\phi$:
\begin{subequations}
\begin{align}
  \ul(\xv,t) &= \int_{t - \delta}^t \sum_{j=1}^M G(\xv - \yv_j, t- \tau) \sigma_j(\tau)[1 - \phi(t - \tau)]\, d\tau,
\label{ul}
  \\
\uh(\xv,t) &= \int_0^t \sum_{j=1}^M G(\xv - \yv_j, t- \tau) \sigma_j(\tau)\phi(t - \tau)\, d\tau,
\label{uh}
\end{align}
\end{subequations}
with $u = \ul + \uh$.
The function $\phi$, defined in the next section, is (numerically) smooth,
with  $\phi(t) = 0$ for $t\leq 0$, and $\phi(t) = 1$ for $t\geq \delta$.
Its parameter $\delta>0$ is
chosen to be of order the shortest timescale present in the signatures $\sigma_j$.
This split is illustrated in Fig.~\ref{fig:split} (specifically the upper region
with $\tau \in [t-\delta,t]$).
The shape of $\phi$ depends on precision parameter $\epsilon>0$,
and is designed so that its $\epsilon$-bandlimit
\footnote{The $\epsilon$-bandlimit of a function $f$ is defined as the
smallest frequency $\omega$ such that $\|\hat f\|_{L^2(\R \backslash [-\omega,\omega])} \le \epsilon$.}
is as small as possible; this will scale as
$\bigO(|\log \epsilon| / \delta)$.
The local part
$\ul$ involves interactions only up to distance $\delta$, so that \eqref{ul} can be computed
directly---%
typically in linear time---using suitable quadrature rules.
The history part $\uh$ is
evaluated as a spectral representation (\eqref{eq:uexactF} modified by a factor $\phi(t-\tau)$ in the integrand) that, due to the blending,
may be truncated with error $\bigO(\epsilon)$ to $|\kv|\le K$ where
$K$ is similar to the signal $\epsilon$-bandlimit.
The fast evaluation of $\uh$ relies on the
non-uniform fast Fourier transform (NUFFT) and overcomes the history-dependence
by exploiting the semigroup property of the governing wave equation.
Specifically, a discretized set of Fourier coefficients
of $\uh$ will be updated by time-marching to arbitrarily high order accuracy,
with a uniform time-step $\dt$
that need only be slightly smaller than the 
time-step needed to resolve the signatures $\sigma_j$ themselves.

\begin{figure}[th]  
  \centering
  \includegraphics[width=0.95\textwidth]{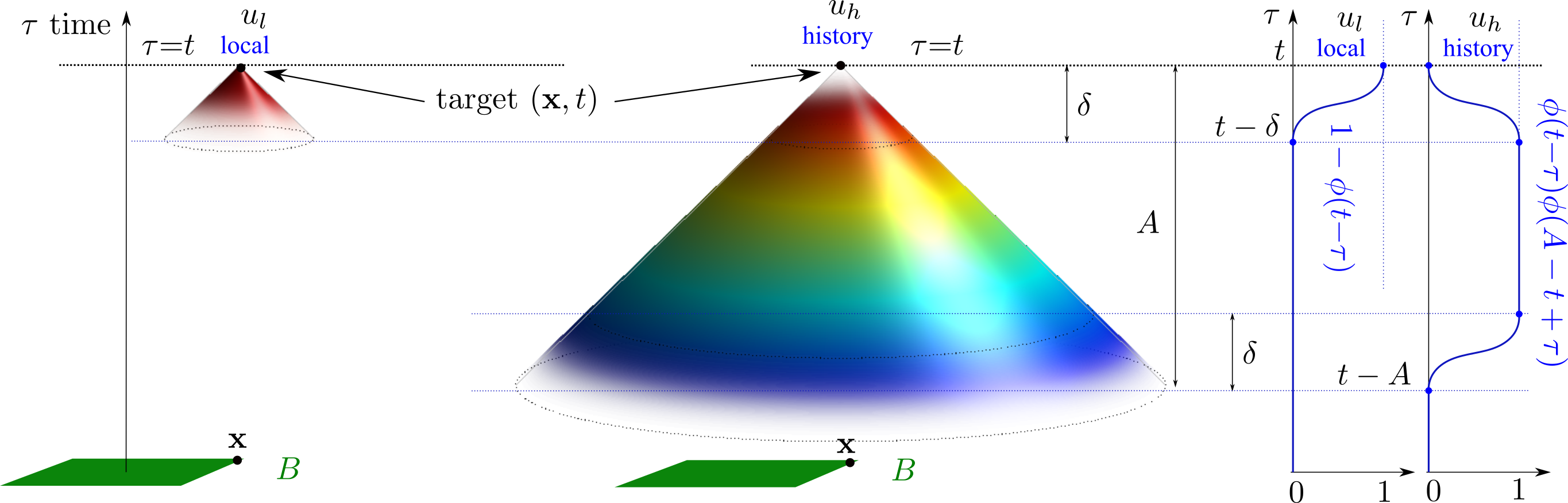}
  \caption{Local and truncated history contributions to the solution $u(\xv,t)$
    at a target location $\xv\in B = [-1,1]^3$ at the current time $t$.
    The left side depicts the space-time diagram (light cone) for the 
    local $\ul$ part, while the central image is the same for
    the history $\uh$ part.
    The vertical axis shows the source time $\tau$ spanning $0$ to
    the current time $t$, and is common to all parts of the figure.
    (Only two of the three dimensions of space are shown; each temporal
    slice of the light ``cone'' is a sphere rather than a circle.)
    The color on the cones $|\xv-\yv|=t-\tau$ indicates the delay $t-\tau$.
    The right side graphs their respective temporal windowing functions;
    see \eqref{ul} for the local part, and \eqref{eq:uexactFtrunc} for the
    history part where the kernel $\CGA$ is smoothly truncated to the finite time horizon
    $t-\tau<A$. Note that
    $A-\delta \ge \diam B$, so that this truncation has no effect in $B$.%
    }
  \label{fig:split}
\end{figure}

In its original form, the WFP method was derived for a periodic box,
avoiding the need to consider the oscillatory nature of the integrand in 
the Fourier integral (the second concern above).
Motivated by the truncated kernel method of \cite{Vico2016} for elliptic PDEs,
we extend the WFP method here by using a
{\it spatially truncated} version
of the wave kernel, which we will denote by $\CGA$.
The essential requirement is that 
$\CGA(\xv-\yv,t)$ coincide with $G(\xv-\yv,t)$
for all $\xv,\yv\in B$, $t\in\R$.
Equivalently, the Green's functions must coincide in a double-sized cubical domain:
\beq
\CGA(\xv,t) = G(\xv,t) \qquad \mbox{for all } \xv \in [-2,2]^3, \; t\in\R.
\label{Gmatch}
\eeq
Radial truncation is extremely convenient since it allows the
passage from spatial to temporal truncation, due to strong Huygens, as follows.
We define 
\beq\label{eq:truncatedKernel}
\CGA(\xv,t) \;:=\; \phi(A - \abs{\xv}) G(\xv,t)    
\; = \; \phi(A-t) G(\xv,t),
\eeq
where we choose the same blending function $\phi$ and same parameter $\delta$
as for the WFP splitting, in order allow the same Fourier truncation $|\kv|\le K$,
and set the history support parameter
\beq
A = 2\sqrt{3} + \delta.
\label{A}
\eeq
One may check (using the fact that $\phi(t)=1$ for $t\ge\delta$)
that this value is the smallest choice of $A$ for which \eqref{Gmatch} holds.
From the last expression in \eqref{eq:truncatedKernel} we immediately see that the (spatial) Fourier
transform of $\CGA(\xv,t)$ is
\beq
\label{eq:truncatedKernelSpectral}
\hat\calG_A(\kv,t) = \phi(A-t) \frac{\sin\kappa t}{\kappa}, \quad t>0.
\eeq
Substituting the truncated kernel for the free-space Green's 
function $G(\xv,t)$ in the Fourier transform domain, using
\eqref{eq:truncatedKernelSpectral}, gives
\beq
\label{eq:uexactFtrunc}
\uh(\xv,t) = \frac{1}{(2\pi)^3} 
\int_{\mathbb{R}^3}\int_{t-A}^t  
e^{-i \kv \cdot \xv} 
\frac{\sin \kappa (t-\tau)}{\kappa} \sighat(\kv,\tau)
\phi(t-\tau)
\phi(A-t+\tau) 
\, d\tau d\kv,
\eeq
which combines the blended ``switch on'' of WFP with the blended ``switch off''
of kernel truncation,
as sketched in the right-most graph in Fig.~\ref{fig:split}.
Note the change in lower time limit, since $\phi(A-t+\tau) = 0$ for $\tau<t-A$:
the time horizon is now finite.
The finite spatial extent of the truncated kernel now
controls the integrand's oscillation rate with respect to $\kv$,
so that using a Nyquist-spaced trapezoid quadrature
needs a wavenumber grid spacing of only $\bigO(1)$,
no matter how large $t$ is.

In short,
the truncated-kernel WFP (TK-WFP) method approximates 
the solution~\eqref{eq:mainSolnRep} to spectral accuracy
by a splitting into a history and local part,
with the history part computed in the Fourier transform domain
using $N = \bigO(1/\dt)$ quadrature nodes per dimension,
and the local
part computed directly. The amount of work required 
is $\Oh{M + N^3\log N}$ per time-step, dominated by the NUFFT.
$\Nt = T/\dt$ such steps are needed to reach the final simulation time $T$.
The storage requirements for the scheme as we present it are greater than
for the original WFP method, since we
will need to keep track of the time history for each source over the horizon
$[t-A,t]$ (see Remark~\ref{r:store}).
The splitting is discussed in detail in section \ref{sec:split}, 
the history part in
sections \ref{sec:history} and \ref{sec:histcomp},
and the local part in section \ref{sec:local}.
Section \ref{sec:results} illustrates the performance of the scheme with several
numerical examples.
We conclude with a brief discussion in Section~\ref{sec:conclusions}.
Two appendices contain the proof of the main theoretical
result (Theorem~\ref{thm:alphaDecay}) on spectral decay of Fourier coefficients.

\subsection{Prior work}

While we do not seek to present a thorough review of the literature, it should be
noted that a variety of algorithms have been developed to rapidly evaluate sums of the
form \eqref{eq:mainSolnRep} at $\Nt$ time steps in 
$\Oh{\Nt M\log^2 M}$ operations. Of particular relevance is the
plane-wave time domain method (PWTD)
of Ergin, Shanker, Michielssen and collaborators \cite{Ergin1998,Liu2020,CMS},
a time-domain version of the fast multipole method. It achieves more or less
optimal asymptotic performance, independent of whether the sources lie on a surface
or are more uniformly distributed in the computational domain. 
In \cite{Meng2010}, Meng et al.\ present an alternative that is easier to implement,
relying on an adaptive Cartesian grid and sparse interpolation methods 
to compute far-field interactions rather than relying on spherical harmonic expansions.
A related acceleration technique is the fast convolution quadrature algorithm
of  \cite{Banjai2014}, which combines the Laplace transform with 
${\cal H}$-matrices and the high-frequency fast-multipole method to achieve the same
complexity. 

Because the associated constants are rather large, however, the design of such fast
algorithms remains an active area of research. One such scheme is
the time-domain adaptive integral method \cite{Yilmaz2004} which, like ours,
is an FFT-based scheme, but does not evolve the Fourier transform of the solution.
Instead relying on the fact that the wave kernel is convolutional in space and time,
it uses multilevel, blocked space-time FFTs.
For layer potentials on surfaces,
the cost is of the order $\Oh{\Nt M^{3/2} \log M}$ operations. 
It is simpler to implement than the PWTD algorithm, with smaller associated constants
but suboptimal asymptotic scaling. The same is true for the method presented
here. 

Recently, hybrid time-frequency approaches have been proposed that can achieve excellent
performance under suitable conditions on the data.
These include \cite{Anderson2020,bruno2025} and the recent work 
of Wilber et al.~\cite{Wilber2025}. The latter combines contour deformation
and the fast sinc transform to overcome the oscillatory nature of the wave kernel
in the Fourier transform domain. Like the scheme of \cite{Banjai2014} mentioned above,
it makes use of fast solvers for a sequence of Helmholtz equations, rather than a 
marching scheme in the time domain. 

There is also, of course, a vast literature on solving the wave equation by direct 
discretization of the governing equation using finite difference
or finite element methods (e.g., \cite{nodalDGbook,kaltenbacher2018computational,Virieux12}),
combined with ``non-reflecting'' or ``absorbing'' boundary conditions
on a finite computational domain. Such boundary conditions are generally designed
as local approximations of the exact non-local exact radiation condition.
Examples of such conditions include those by Engquist and Majda~\cite{Engquist1977}, 
Bayliss and Turkel in~\cite{Bayliss1980}, and Higdon~\cite{Higdon1990}. 
These methods are generally low order accurate and tend to lose precision for waves
that are not normally incident to the artificial boundary.
For greater precision, absorbing region methods have been developed
that minimize numerical reflection by modifying the governing 
PDE itself by adding a damping term~\cite{Israeli1981}, or by a 
coordinate transformation, such as the Perfectly Matched Layer (PML) method~\cite{Berenger1994}. We refer the reader to the review articles
\cite{Givoli2008,Hagstrom1999,tsynkovrev} for further details.
Ongoing work in that direction
includes ideas from numerical relativity \cite{zenginoglu11} and geophysics \cite{kahana22}.
More recently, the double absorbing boundary (DAB) method and 
complete radiation boundary conditions (CRBC) were introduced by 
Hagstrom, Warburton and Givoli~\cite{Hagstrom2004, Hagstrom2009, Hagstrom2010,Potter2015}, 
which are able to accurately approximate the exact radiation condition to a 
specified error tolerance. They are especially convenient for coupling to finite
difference methods because they can be applied on rectangular boxes of arbitrary 
aspect ratio.
Exact outgoing conditions have also been developed for spherical boundaries 
\cite{Alpert2000,grotekeller}.

Finally, we should note the methods of 
Global Discrete Artificial Boundary Conditions \cite{tsynkov01}
and Time-Dependent Phase Space Filters~\cite{SOFFER2007, SOFFER2009, SofferStucchio}.
The former relies on the strong Huygens' principle
and requires the solution of an auxiliary problem on a larger domain. 
The latter detects microlocally outgoing 
components of the solution and deletes them from the representation using 
the Fourier transform at periodic intervals in time. Both can be carried out with 
high precision.

Note that our proposal's cost per timestep of $\bigO(N^3 \log N)$, $N$ being the highest frequency,
is, up to the log factor, the same order as that of a finite difference or element discretization
with a fixed number of grid-points per wavelength.
The latter is well known to suffer from dispersion errors unless the order is very high
\cite{bayliss85,Virieux12} \cite[\S4.6]{nodalDGbook}.
A major advantage of our method is that it achieves global spectral accuracy, without dispersion
error, through the use of the discretized Fourier transform.

\section{Windowed Fourier projection and the decomposition of hyperbolic potentials} \label{sec:split}

The windowed Fourier projection method (WFP) \cite{wfp2025}
is based on
a numerically smooth partition of unity using a blending function $\phi(t)$, with
$\phi(t) = 0$ for $t \leq 0$ and $\phi(t) = 1$ for $t \geq \delta$, with
the small parameter $\delta>0$ chosen of order the time-step.
Its precise value will be adjusted to balance
the local and history computational effort.

\subsection{The blending function}
\label{sec:blending}

As in \cite{wfp2025}, we exploit the
Kaiser--Bessel ``bump'' function $I_0(b\sqrt{1 - t^2})$, $t\in[-1,1]$, where $I_0$ 
is the zeroth-order modified Bessel function, while $b = \ln(1/\epsilon)$ is a 
shape parameter. This bump function has many applications in signal and image processing.
Its peak width scales as $b^{-1/2}$ while its peak value scales as $e^b$.
This is shifted and scaled to a unit-integral bump on the domain $[0,\delta]$, and its
antiderivative is taken:
\beq\label{eq:KBblending}
\phi(t) := \int_0^t\phi'(\tau)d\tau, \quad \text{where}\quad \phi'(t):= \begin{cases}\frac{b}{\delta\sinh b}I_0\left(b\sqrt{1 - (2t/\delta - 1)^2}\right), & 0\leq t\leq \delta, \\
0, & t<0\text{ or }t>\delta.\end{cases}
\eeq
This satisfies the following properties.
\begin{enumerate}
\item
The function $\phi$ transitions smoothly form $0$ to $1$ in $[0,\delta]$.
It is formally non-smooth on $\R$, there being
an $\Oh{e^{-b}}$ jump in $\phi'$ at $t = 0$ and $t = \delta$.
Thus the choice $b = \ln(1/\epsilon)$, for $\epsilon\ll 1$,
ensures that $\phi$
is \textit{numerically} smooth up to tolerance $\epsilon$. \label{itm:prop1}
\item The Fourier transform
of $\phi'$ is available analytically:
\beq
\hat{\phi'}(\omega) := \int_\mathbb{R}\phi'(t)e^{i\omega t}dt, 
= \frac{be^{-i\delta\omega/2}}{\sinh b}\sinc \sqrt{\left(\frac{\delta\omega}{2}\right)^2 - b^2}, \qquad \omega\in\mathbb R,
\label{eq:windowFT}
\eeq
where $\sinc(z) := \sin(z)/z$ for $z\neq0$ and $1$ otherwise.
The $\bigO(\epsilon)$-support of $\hat{\phi'}$ is $[-2b/\delta,2b/\delta]$,
as one may see by noting that the argument of the $\sinc$
is real outside of this domain (the $\sinc$ is exponentially large in the domain).
In other words, the $\epsilon$-bandlimit of $\phi'$ is $2b/\delta$.
\label{itm:prop2}
\end{enumerate}

Elaborating on the second property, we will need the following decay bound beyond the bandlimit, which is proved in \ref{sec:windowFourierProof}.
\begin{lem}\label{thm:windowFourierLem}
  Let $\delta>0$ and $0<\epsilon<1$.
  Then \eqref{eq:windowFT} with $b = \ln(1/\epsilon)$ obeys, for any
  $\theta>1$, the bound
\beq\label{eq:KB_bandlimit}
\bigl|
\hat{\phi'}(\omega)
\bigr|
< \frac{4b\theta}{\delta}\cdot\frac{\epsilon}{\abs{\omega}},
\qquad \text{for all }
\;\;
|\omega|\geq\frac{2b}{\delta\sqrt{1 - 1/\theta^2}}. 
\eeq
\end{lem}

\subsection{The local and history part}

Recall that the solution is represented as $u(\xv,t) = \ul(\xv,t) + \uh(\xv,t)$,
with the local part defined by \eqref{ul} and the history part
\beq\label{eq:historyPart}
\uh(\xv,t) =\sum_{j = 1}^M\int_{0}^t\calG_A(\xv - \yv_j,t - \tau)\sigma_j(\tau)\phi(t - \tau)d\tau,
\eeq
where $\CGA$ is the truncated kernel \eqref{eq:truncatedKernel}.

The local part \eqref{ul} involves an integral 
over the most recent times, and can be treated directly 
as in \eqref{eq:mainSolnRep2b} but for a spatially restricted region around each target:
\beq\label{eq:local}
\ul(\xv,t) = \sum_{j \in\calNd(\xv)}\frac{\sigma_j(t - r_j)[1 - \phi(r_j)]}{4\pi r_j},
\eeq
with the abbreviation $r_j: = \abs{\xv -  \yv_j}$, and the neighbor
index set definition
$
\calNd( \xv) := \{\ j \ | \ 0<r_j<\delta, \ j = 1, 2,\dots, M\}.
$
Note that we use $G(\xv,t)$ in \eqref{ul} since
for short times
it is equal to $\calG_A(\xv,t)$, for all $\xv\in B$,
because $\delta<A-\delta$.
The high-order
quadrature approximation of $\ul$ is then routine, and we postpone it
to Section~\ref{sec:local}.

The history part of the solution, $\uh$, on the other hand,
will be expressed as the Fourier integral
\beq\label{eq:uh_spectral}
\uh(\xv,t)  = \frac{1}{(2\pi)^3} \int_{\mathbb R^{3}} \alpha(\kv,t) e^{-i\kv\cdot \xv}d\kv, 
\eeq
with coefficients from \eqref{eq:uexactFtrunc} given by
\beq\label{eq:alphak}
\alpha(\kv,t) = \int_{t - A}^{t}\frac{\sin \kappa (t - \tau)}{\kappa}\sighat(\kv,\tau)\phi(t - \tau)\phi(A - t + \tau) d\tau,
\eeq
recalling the Fourier source function $S(\kv,\tau)$ defined in \eqref{eq:sighat}.
Its numerical approximation and fast computation occupies the next two sections.

\section{Approximation of the history part and choice of $\delta$} \label{sec:history}

\subsection{Discretization of the Fourier integral}

The Fourier integral \eqref{eq:uh_spectral} for $\uh$ may be discretized
using the 3D infinite trapezoidal rule with equal step size $\dk$:
\beq\label{eq:uh_discr}
\uh(\xv,t) = \frac{1}{(2\pi)^3} \int_{\mathbb R^{3}} \alpha(\kv,t) e^{-i\kv\cdot \xv}d\kv
\approx
\left(\frac{\dk}{2\pi}\right)^3 \sum_{\nv\in\mathbb Z^{3}}
\alpha(\nv\dk,t)e^{-i\nv\dk\cdot\xv},
\quad \xv\in B.
\eeq
Deferring the issue of truncation of this infinite sum to the next subsection,
we first discuss the discretization error.
Remarkably, once $\dk$ is below an explicit $\bigO(1)$ constant, there is no such error.
\begin{prop}
  Let $\dk \le 2\pi/(A+2)$, recalling \eqref{A}, and let the coefficients $\alpha(\kv,t)$ be given
  by \eqref{eq:alphak} with $S(\kv,\tau)$ given by \eqref{eq:sighat}.
  Then the discretization error is zero,
  that is, \eqref{eq:uh_discr} is an equality.
\end{prop}
The proof relies on the finite spatial support of $\CGA$, as follows.
Shifting and scaling the 3D version of the Poisson summation formula
(e.g., \cite[Ch.~VII, Cor.~2.6]{steinweissbook}) gives
\[
\left(\frac{\dk}{2\pi}\right)^3
\sum_{\nv\in\mathbb Z^{3}}
\alpha(\nv\dk,t)e^{-i\nv\dk\cdot\xv}
\; =\; \uh(\xv,t) \; + \!\!
\sum_{\mv\in\mathbb Z^{3} \backslash \{{\bf 0}\}} \uh\biggl(\xv + \frac{2\pi}{\dk}\mv, t\biggr).
 \]
 The error is the 2nd term on the right-hand side.
 Since $\yv_j\in B$ for all $j$, while the spatial support of $\CGA$
 (recalling \eqref{eq:truncatedKernel}) lies within
 the ball of radius $A$, the spatial support of $\uh(\cdot,t)$ lives within the cube $[-A-1,A+1]^3$.
 Yet for $\xv\in B$, each spatial argument $\xv + (2\pi/\dk) \mv$ lies outside of this
 cube, since $2\pi/\dk \ge A+2$. Thus each term on the right-hand side sum is
 zero, and so is the error.
 
 This result may also be interpreted as the
 exactness of the periodic trapezoid quadrature when applied
 to an integrand that has bandlimit (with respect to each component of $\kv$)
 no more than $2\pi/\dk$.
 (For the periodic case see, for example, \cite[Cor.~3.3]{Trefethen2014}.)
Given the above, in practice we set
\beq
\dk \approx \frac{2\pi}{A + 2},
\label{dkprac}
\eeq
which is about $1.15$ when $\delta\ll 1$.

\subsection{Wavenumber truncation}

Next we determine a cut-off wavenumber magnitude $K$ such that
truncation of the above sum to the ball $|\kv|\le K$ incurs error of only $\bigO(\epsilon)$, where $\epsilon$ is a user-chosen tolerance.
This is supplied by the following.

\begin{theorem}\label{thm:alphaDecay}
  Let $\sigma_j\in L_2(\mathbb R_+)$,
  with $\|\sigma_j\|_1\le P$,
  $j=1,\dots,M$, be given source signature functions.
  Let the tolerance $0<\epsilon < 1$ and numerical bandlimit $K_0$ be
  such that the Fourier transforms $\hat{\sigma}_j(\omega)$ decay as
\beq
\abs{\hat{\sigma}_j(\omega)} \leq \frac{\epsilon}{\omega^2},
\quad \text{for all }|\omega|>K_0.
\label{sighatdecay}
\eeq
Let the blending timescale be $\delta>0$,
then let $\phi$ be defined as in \eqref{eq:KBblending}
with $b = \ln(1/\epsilon)$. Then, for each $\theta>1$,
the Fourier coefficients defined by \eqref{eq:alphak} obey the decay condition
\beq
\abs{\alpha(\kv,t)} = \frac{C M \epsilon}{\kappa^3},
\qquad \text{for all }\kappa>K := K_0 + \frac{2b}{\delta\sqrt{1 - 1/\theta^2}},
\; t\in[0,T],
\label{albnd}
\eeq 
recalling the notation $\kappa:=|\kv|$.
Here $C$ is some constant independent of $\epsilon$,
that depends only weakly on $K_0$, $b$, $\delta$, $\theta$, and $P$.
\end{theorem}
This implies that the coefficients $\alpha(\kv,t)$ are of size $\bigO(\epsilon)$,
with rapid algebraic decay, beyond a numerical wavenumber cut-off $K$.
This cut-off need only be slightly larger than
the $\epsilon$-bandlimit $K_0$ of the
given signals plus the $\epsilon$-bandlimit $2b/\delta$ of the blending function.
(By choosing $\theta$ sufficiently large, $K$ may be pushed arbitrarily close
to this sum at the cost of only mild growth in the error prefactor.)

The proof of the above theorem is given in \ref{sec:alphaDecayProof},
which also states the explicit formula for $C$ in \eqref{albnd}.
The proof combines the decay of the signals
outside their $\epsilon$-bandwidth $K_0$, the decay of $\hat{\phi'}$ outside
its $\epsilon$-bandwidth (Lemma~\ref{thm:windowFourierLem}),
the convolution theorem, and the fact that the wave kernel
$\sin[\kappa(t-\tau)]/\kappa$ converts temporal to spatial bandlimitness.
The technical condition \eqref{sighatdecay} is merely a convenient expression
of the signal being $\epsilon$-bandlimited; Gaussians and other common
smooth signals will have tails that decay faster than quadratically in $\omega$.
Note that, although one could prove a much simpler theorem by assuming
that $\sigma_j$ is \textit{strictly} bandlimited (hence analytic on $\R$),
this assumption would only admit trivial data $\sigma_j\equiv 0$ via
the need for $\sigma_j(t)=0$ for all $t<0$; such a theorem would not be interesting!

\begin{remark}[Bound on the truncation error]
  Although Theorem \ref{thm:alphaDecay} shows that the coefficients $\alpha(\kv,t)$
  are uniformly $\bigO(\epsilon)$ in $|\kv|\ge K$, unfortunately
  their $1/\kappa^3$ wavenumber decay rate is not quite fast enough
  to prove summability of the error
 $\sum_{\nv\in\Z^3: |\nv\dk|>K} \alpha(\nv\dk,t) e^{-i\nv\dk\cdot\xv}$
  induced by truncation in $d=3$
  dimensions, at least not via a bound on $|\alpha|$.
  This limitation is due to the $\epsilon/|\omega|$ tail bound for
  $\hat{\phi'}$, in turn tied to the $\bigO(\epsilon)$
  discontinuities in $\phi'$ at $0$ and $\delta$.
  One may be able to fix this theoretical deficit by switching to the ``deplinthed''
  (continuous) Kaiser--Bessel window as in \cite{dftsubmat}.
  However, the numerical performance is expected to be almost identical.
\end{remark}

Motivated by the above theorem, we thus truncate the infinite sum \eqref{eq:uh_discr}
to the finite ball $|\kv|\le K$, whose radius in practice we set as
\beq
K = K_0 + \frac{2b}{\delta},
\label{Kprac}
\eeq
the large-$\theta$ limit of \eqref{albnd}. Thus, combining the results
of this section so far, with the above choices of
$\dk$, $b$, and $K$, we expect that
$\uh$ is approximated to error $\bigO(\epsilon)$ by the discrete sum
\beq
\uh(\xv,t)
\approx
\left(\frac{\dk}{2\pi}\right)^3 \sum_{\nv\in\Z^3: \, |\nv \dk| \le K}
\alpha(\nv\dk,t)e^{-i\nv\dk\cdot\xv},
\quad \xv\in B.
\label{eq:uh_approx}
\eeq

\subsection{Choice of blending timescale}

Given the tolerance $\epsilon$ and the signal $\epsilon$-bandwidth $K_0$,
it remains to set the blending parameter $\delta$.
As in \cite{wfp2025}, we set $\delta = W\dt$,
where $\dt$ the time-step with which the history coefficients
are evolved, and $W$ is a small integer.
We choose
\beq
W:=\biggl\lceil \frac{2b}{\pi\gamma}\biggr\rceil
,
\label{W}
\eeq
where $\gamma\in(0,1)$ is a dimensionless parameter,
and note that $W$ grows as $\log(1/\epsilon)$ as more precision is demanded.
We then have $2b/\delta \approx \pi \gamma/\dt$.
Thus \eqref{Kprac} becomes $K \approx K_0 + \frac{\pi \gamma}{\dt}$.
This indicates that $\gamma$ is the fraction of the time-step Nyquist bandlimit
$\pi/\dt$ sacrificed to the blending function,
the remaining fraction being accounted for by the signal bandlimit $K_0$.
Note that reducing $\gamma$ decreases $K$ and, therefore,
the cost associated with the history part,
but increases $W$ and $\delta$ and, therefore,
the amount of work required in the local part. Thus, $\gamma$ can be used to adjust
the history/local workload balance.
We typically choose $\gamma \approx 0.5$.

Note also that if $\dt$ were set much smaller than the Nyquist value for the signal,
$\pi/K_0$, the bandlimit (hence cost) of the history part would be excessively large.
Fortunately, with a high order method, the solution converges rapidly once
$\dt < \pi/K_0$. Setting $\dt = \pi/(2K_0)$, for example, with $\gamma=0.5$,
results in $K = 2K_0$, an acceptable increase in the bandlimit.

\section{Computation of the history part} \label{sec:histcomp}

The reason that the history part can be evaluated efficiently is that
the Fourier coefficients $\alpha(\kv,t)$ satisfy evolution formulae and 
can be computed recursively at each time step.
In particular, it is easy to check that, independently at each wavenumber $\kv$,
they satisfy the ODE initial value problem
\beq\label{eq:alphakODE}
\begin{cases}
\ddot\alpha(\kv,t) + \kappa^2\alpha(\kv,t) = F(\kv,t), & t>0, \\
\alpha(\kv,0) = \dot\alpha(\kv,0) = 0, &
\end{cases}
\eeq
using the notation $\dot\alpha := \partial_t \alpha$, with forcing
\beq\label{eq:alphakF}
F(\kv,t) = \int_{t - \delta}^{t}\left[ \Psi(\kv,t - \tau)\sighat(\kv,\tau) - \Psi_A(\kv,t - \tau)\sighat(\kv,\tau - A + \delta)\right]d\tau ,
\eeq
where
\beq
\begin{split}
\Psi(\kv,\tau)&:= 2\cos\kappa\tau \phi'(\tau) + \frac{\sin\kappa\tau}{\kappa}\phi''(\tau), \\
\Psi_A(\kv,\tau)&:= 2\cos\kappa(\tau + A - \delta) \phi'(\tau) + \frac{\sin\kappa(\tau + A - \delta)}{\kappa}\phi''(\tau).
\end{split}
\eeq
This is done for each discrete $\kv = \nv \dk$ with $|\kv|\le K$, as in the previous section.
In these and the below expressions, the case $\kv={\bf 0}$ 
is taken as the limit $\kappa \to 0$, so that $(\sin \kappa \tau)/\kappa$
becomes $\tau$, etc.

\begin{remark}\label{r:crea}
We refer to $\Psi$ as the {\em creation} 
kernel, since it is responsible for injecting contributions to $\alpha(\kv,t)$
near the current time $t$ to balance the disappearing local part;
this is as in \cite{wfp2025}.
We refer to the new term $\Psi_A$ as the {\em annihilation} kernel since it is responsible for removing 
the historical contributions to $\alpha(\kv,t)$ from times earlier than $t-A$.
\end{remark}

The solution to \eqref{eq:alphakODE} at each $\kv$ is
\beq
\alpha(\kv,t) = \int_0^t \frac{\sin \kappa (t-\tau)}{\kappa} F(\kv,\tau) d\tau.
\label{alF}
\eeq
One may thus interpret \eqref{eq:alphakF}
as the driving such that the coefficients $\alpha$ obey the simple 2nd-order 
ODE~\eqref{eq:alphakODE}, enabling the standard Duhamel-type solution~\eqref{alF}
in contrast to the windowed evaluation formula~\eqref{eq:alphak}.

The following can be easily derived from \eqref{alF}, or from the Duhamel principle applied
to the ODE~\eqref{eq:alphakODE}.
\begin{lem}
Let $\dt>0$ denote a time step. The Fourier coefficients $\alpha(\kv,t)$ satisfy the evolution fomulae
\beq\label{eq:alphakEvolution}
\begin{split}
\alpha(\kv,t + \dt) \;&=\; \alpha(\kv,t)\cos(\kappa\dt) + \dot \alpha(\kv,t) \frac{\sin\kappa\dt}{\kappa} + h(\kv,t), \\
\dot\alpha(\kv,t + \dt) \;&=\; -\kappa\alpha(\kv,t)\sin(\kappa\dt) + \dot \alpha(\kv,t) \cos\kappa\dt+ g(\kv,t), \\
\end{split}
\eeq
where
\beq
h(\kv,t) := \int_t^{t + \dt} \frac{\sin\kappa(t + \dt - \tau)}{\kappa} F(\kv,\tau)d\tau,
\quad 
g(\kv,t) := \int_t^{t + \dt} \cos\kappa(t + \dt - \tau) F(\kv,\tau)d\tau.
\label{hg}
\eeq
\end{lem}
Thus by storing the pair $\alpha$ and $\dot\alpha$
at each $\kv$, their exact update over one time step
takes the form of a $2\times 2$ matrix-vector multiply plus a known vector.

Gauss--Legendre quadrature over the interval $[t,t+\dt]$
can be used to approximate $h$ and $g$ in \eqref{hg} to high order accuracy.
In turn the integral over $[t-\delta,t]$ needed for $F$ in \eqref{eq:alphakF}
may be approximated using the trapezoid rule on the time-stepping grid itself.
By exchanging the order of these integrals, one may precompute weights
$p_m(\kv)$ and $q_m(\kv)$ such that for a current time $t$ on the time grid,
\beq
\begin{aligned}
  h(\kv,t) &\approx \dt \sum_{m=0}^{W-1} p_m(\kv) \sighat(\kv,t-m\dt)
  - p^{(A)}_m(\kv) \sighat(\kv,t-A+\delta-m\dt),
  \\
  g(\kv,t) &\approx \dt \sum_{m=0}^{W-1} q_m(\kv) \sighat(\kv,t-m\dt)
  - q^{(A)}_m(\kv) \sighat(\kv,t-A+\delta-m\dt),
\end{aligned}
\label{hggrid}
\eeq
holds to high order accuracy.
See \cite[\S3]{wfp2025} for the formulae for $p$ and $q$;
the formulae for $p^{(A)}$ and $q^{(A)}$ are analogous.
Values of $\sighat$ for $t<0$ are assumed to be zero.

\subsection{Computation and storage costs}

At each time step, $\sighat(\kv,t)$ must be computed
from \eqref{eq:sighat}, for all wavevectors on the grid $\kv = \nv \dk$, $\nv\in\Z^3$,
with $|\kv|\le K$.
This is approximated to relative tolerance $\epsilon$ by a single application of a NUFFT
(Non-Uniform fast Fourier transform) of type 2
\cite{finufft,finufftlib}, outputting the cube of $N^3$ coefficients
where $N := \lceil \frac{2K}{\dk} \rceil$ is the maximum number of wavevector quadrature points
per dimension.
The values with wavevectors in this cube but outside the ball $|\kv|=K$ are then set to zero.
This NUFFT requires
$\Oh{\log^3(1/\epsilon) M + N^3\log N}$ work.
The subsequent calculation of $h(\kv,t)$ and $g(\kv,t)$ at each time step using \eqref{hggrid}
takes a linear combination of $\sighat(\kv,\tau)$ for $\tau$ ranging over
the most recent $W$ time-steps and the $W$ time steps preceding $t-A$.

The Fourier coefficients $\alpha(\kv,t)$ are then updated via
\eqref{eq:alphakEvolution}. Finally, the evaluation of $u_h$ 
via \eqref{eq:uh_approx}
at $N_x$ target points $\{\xv_i\}_{i=1}^{N_x}$ takes the form of a type 1 NUFFT, which can
be performed at a cost $\Oh{\log^3(1/\epsilon)N_x + N^3\log N}$ at each time step.
Again, here the $\alpha(\kv,t)$ values for $|\kv|>K$ are padded to zero, since this NUFFT
acts on a $N^3$ cubical input coefficient array.

\begin{remark}[History storage costs]
  \label{r:store}
Since the annihilation kernel is applied to historical data up to time $A$ in the past,
we need to store either  $\sighat(\kv,\tau)$ for all time steps in the interval $[t-A,t]$, or the
source signals $\sigma_j(\tau)$, from which 
$\sighat(\kv,\tau)$ can be reconstructed by a NUFFT on the fly.
Data from time steps before time $t-A$ can be safely deleted.
We note that regrouping via a partition of unity in time
(as in, e.g., \cite{tsynkov01})
could much reduce this storage, at the cost of a slight increase in $A$.
\end{remark}

\section{Computation of the local part} \label{sec:local}
In this section we discuss the numerical computation of the local part of the solution, $\ul$, in~\eqref{eq:local}:
\beq
\ul(\xv_i,t) = \sum_{j \in\calNd(\xv_i)}\frac{\sigma_j(t - r_{ij})[1 - \phi(r_{ij})]}{4\pi r_{ij}},
\eeq
for target points $\xv_i$, $i = 1, \dots, \Nx$, and source points $\yv_j$, 
$j = 1, \dots, M$. Recall that
$r_{ij} = \abs{\xv_i - \yv_j}$, and $\calNd(\xv_i)$ denotes the set of source indices
with distances in $(0,\delta)$ from $\xv_i$;
this excludes any sources that coincide with the target.
Note that if $\xv_i=\yv_j$ for some source $j$,
the above expression results in a total potential $u = \ul+\uh$
in \eqref{eq:mainSolnRep2b} which excludes the self-interaction.
Assuming that the densities $\sigma_j$ are available at all retarded time values, $(t - r_{ij})$, we seek to express $\ul$ at each target point as a dot product of two vectors. For this, we precompute the set 
\beq
Q_{ij} = \frac{1 - \phi(r_{ij})}{4\pi r_{ij}}, \quad i = 1, \dots, \Nx, \ j \in\calNd(x_{i}).
\eeq
Then, at each new time $t$, these are used to weight the signal samples:
\beq\label{eq:ul_compute}
\ul(\xv_i,t) = \sum_{j\in\calNd(\xv_i)}\sigma_j(t - r_{ij})Q_{ij}, \qquad i = 1, \dots, \Nx.
\eeq
The cost of computing $\ul$ at $\Nx$ targets and $M$ sources scales linearly with
the typical number of sources within a $\delta$-neighborhood of each target;
the latter we denote by $\Mtyp$.
The local cost for evaluation at a single time $t$ is thus $\Oh{\Nx\Mtyp}$ flops.
For a quasi-uniform volumetric distribution of sources and targets,
$\Mtyp \approx 4\pi\delta^3 M/(3V_B)$, recalling the volume of the sphere of radius $\delta$, and using $V_B = 8$, the volume of the computational domain $B$.
For a quasi-uniform distribution on a surface, $\Mtyp$ instead scales as $\delta^2M$ for target points on the surface (relevant for the time-domain integral equation setting), while for distant off-surface evaluation one may have $\Mtyp=0$.

\begin{remark}
  To find sources within $\calNd(\xv_{i})$ of a given target $\xv_i$ in a fast manner, one may sort source and target points into cubic boxes of side-length $\delta$ and search neighboring boxes only. For example, to evaluate $\ul$ at a target point $\xv_i$ in box $b_i$, we find sources in $\calNd(\xv_i)$ by searching the boxes that share a vertex, edge, or a face with $b_i$. This reduces the cost of the search for filling all elements of $Q$
  from $\Oh{\Nx M}$ to $\Oh{\delta^3\Nx M}$. 
\end{remark} 

\begin{remark}
  For time-domain wave scattering schemes, $\sigma_j$ may only be available on the uniform time grid. In this case, high-order interpolation from the uniform time grid is needed to
  approximate $\sigma_j(t - r_{ij})$, as in \cite{wfp2025}. In the present paper,
  for simplicity, we assume that $\sigma_j$ is a given function that may be rapidly evaluated
  at any argument.
\end{remark}

\section{Numerical results} \label{sec:results}

We present numerical examples to verify the high-order accuracy and computational performance of the TK-WFP algorithm. 
In Section~\ref{sec:eightSources}, we demonstrate exponential convergence with respect to $\dt$, using a small number of source points.
In Section~\ref{sec:cruller}, we investigate the algorithm's performance in wave-scattering scenarios with many sources and targets covering a closed surface.
Section~\ref{sec:randomSources} explores the accuracy and speed of the TK-WFP solution approach to a problem with a wide range of frequency settings and a large number of sources and targets.

In each of the numerical examples, we approximate the wave equation solution $u$ to \eqref{eq:freeSpace} using TK-WFP with a constant step size $\dt$.
We denote the resulting approximate solution by $\tilde{u}$.
We fix the error tolerance $\epsilon = 10^{-6}$.
We measure the error at a given time $t$ by evaluating the exact solution $u$ directly using
\eqref{eq:mainSolnRep2b}.
We will report the maximum norm of the absolute, or relative error, defined respectively as
\beq
\calE_{\dt} := \norm{\tilde u - u}_{\infty}, \qquad \tilde\calE_{\dt} := \frac{\calE_{\dt}}{\norm{u}_\infty}, 
\eeq
for a particular $\dt$ parameter.
The norms are estimated as the maximum over a large number
of target points $\xv_i$ at the final time $t = T$. 

We clarify that evaluating the solution $u$ at the last time slice $t=T$ requires computing the Fourier coefficients $\alpha(\kv,n\dt)$, for each $\abs{\kv} \le K$, on the regular grid $n = 0, 1, \dots, \Nt$, recalling $\Nt\dt = T$.
We thus differentiate between \textit{time steps} (of size $\dt$), with which these coefficients are evolved, and \textit{time slices} $t$, where we choose to evaluate the full
solution via $u=\ul + \uh$.

The TK-WFP and direct evaluations were implemented in MATLAB (version
R2023b), with NUFFTs performed by calling the MATLAB interface to
the C++ library FINUFFT (version 2.4.1). MATLAB sparse matrices were used to
manipulate the set of entries $Q_{ij}$ for the local part evaluation.

\subsection{Convergence study}\label{sec:eightSources}
We begin with a simple example where we evaluate the solution representation $u$ in~\eqref{eq:mainSolnRep} with a few point sources using the TK-WFP method,
to test that the free-space radiative boundary conditions are correct.
Let the source points $\yv_j$ be located at each of the eight corners $(\pm1,\pm1,\pm1)$
of the cube $B$;
this excites the most severe grazing incident waves
(these are known to be problematic when using approximate radiation boundary conditions).
We choose identical time signatures
\beq
\sigma_j(t) = 0.5\left[\erf(5(t - 1.5)) + 1\right]\sin(30\pi(t-1.5)), \qquad j = 1, 2, \dots, 8,
\eeq
which take the form of a fixed high frequency at $30\pi$ modulated by a smooth
``ramp up'' erf blending function.
At a tolerance $\epsilon=10^{-6}$ the $\epsilon$-bandlimit is $K_0 \approx 131$;
this is $30\pi$ plus the $\epsilon$-bandlimit of the erf.
For the latter we use the estimate that the $\epsilon$-bandlimit of $\erf(\mu t)$ is
around $2 \mu \sqrt{\log(1/\epsilon)}$, derived by the Fourier transform of
its derivative (the Gaussian $e^{-(\mu t)^2}$, up to an $\bigO(1)$ prefactor).
At the center signal frequency $30\pi$, the size of the cube is $30$ wavelengths on a side;
at the bandlimit $K_0$, this is $42$ wavelengths.
We fix the Nyquist fraction for the blending $\gamma = 0.5$. The width \eqref{W} of the blending region is then $W = 18$ time steps, recalling that $\delta = W\dt$.

Figure~\ref{fig:exp6} (left) shows the computed solution $\tilde u$ at a $100\times 100 \times 100$ mesh of target points uniformly covering $B$, at a final time $T = 6$, and with time step $\dt =0.0102$.
Figure~\ref{fig:exp6} (right) presents a 2D slice of that solution at $z = -0.01$. We observe clean outgoing waves at the boundary of $B$.
The absolute error of this solution is
$\calE_{\dt} =9.8\times10^{-7}$.

In Figure~\ref{fig:exp6_conv} we show, for the same example,
convergence with respect to $\dt$ of the absolute error.
The error is estimated by evaluating $\tilde u(\xv,t)$ over a $10\times 10\times 10$
uniform mesh of target points covering $B$, at $t=6$.
At each $\dt$, the maximum wavenumber is set as $K=\pi/\dt$, rather than
fixing $K$ using~\eqref{Kprac}. The cut-off distance is $\delta=0.1831$.
The plot demonstrates spectral convergence of the TK-WFP method, down to
saturation around the chosen 6-digit tolerance.
The number of sources is too small to report meaningful CPU timings;
for that we move to the following larger examples.

\begin{figure}[th]
	\centering        
	\includegraphics[width=7cm]{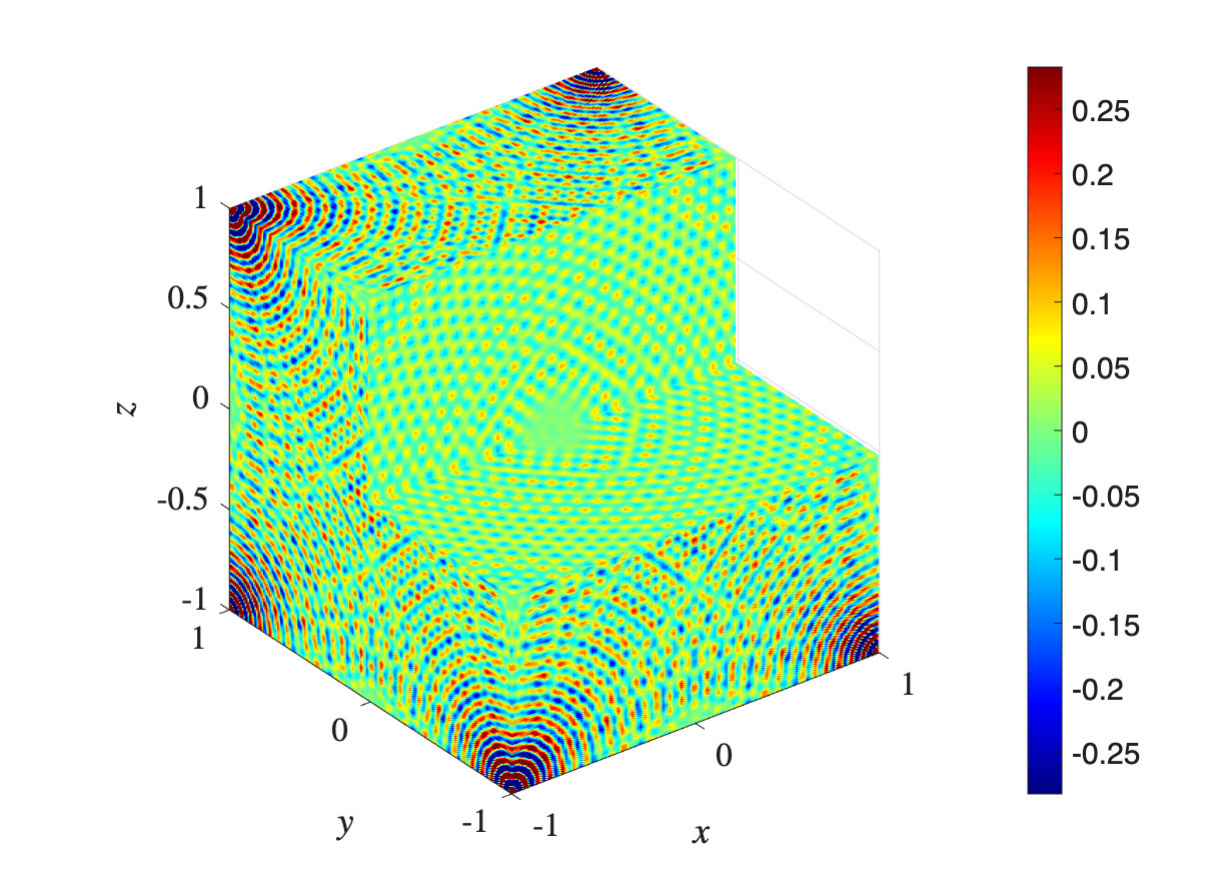}
        \includegraphics[width=7cm]{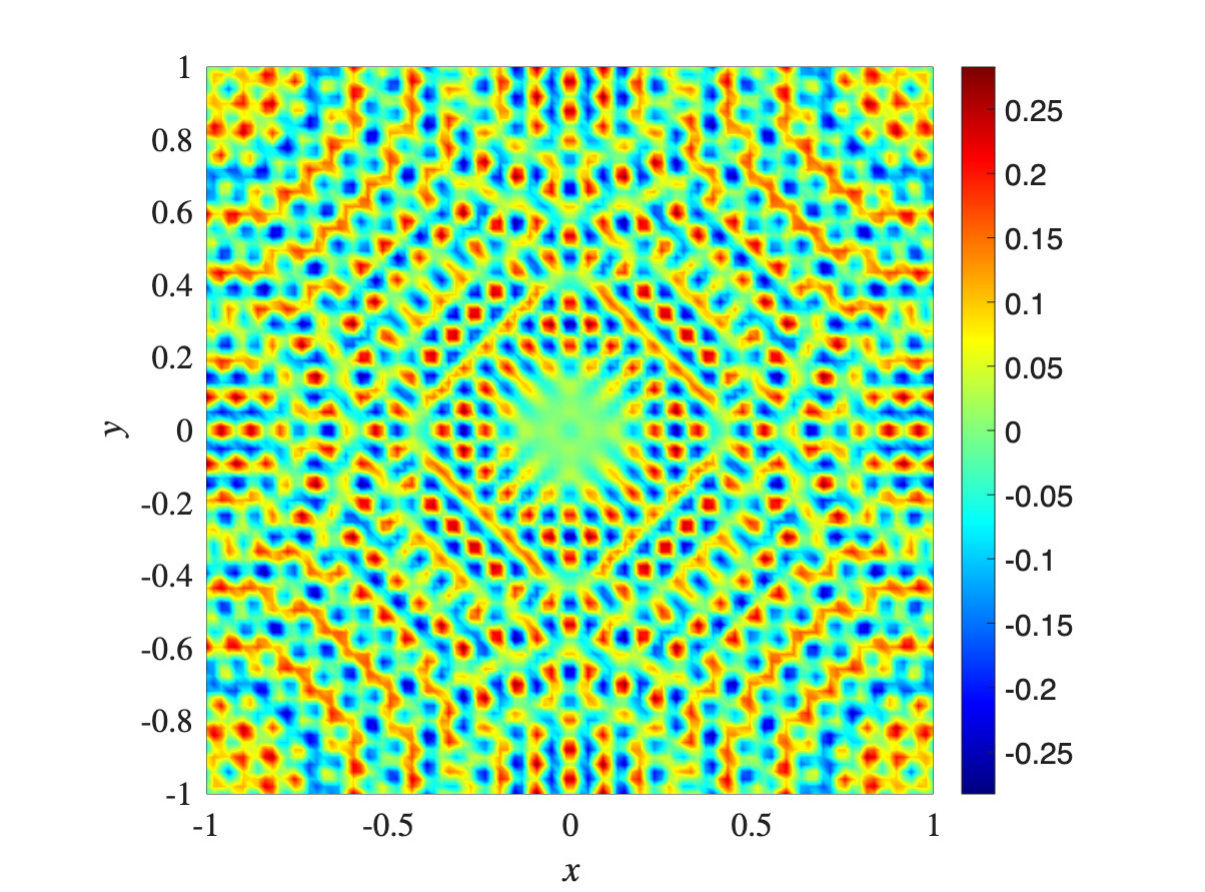}
	\caption{Computed solution due to eight sources as in Section~\ref{sec:eightSources},
          at time $t = 6$, using $\dt=0.0102$. The same potential $\tilde u$ is shown as a cut-away of the volume (left), and on the plane $z = -0.01$ (right). At its center frequency (visible as the
          oscillations in the plot) the wave field has 30 wavelengths on each side of the cube $B$. The solution has 6 accurate digits.}
	\label{fig:exp6}
\end{figure}

\begin{figure}[th]
	\centering
        \includegraphics[width=8cm]{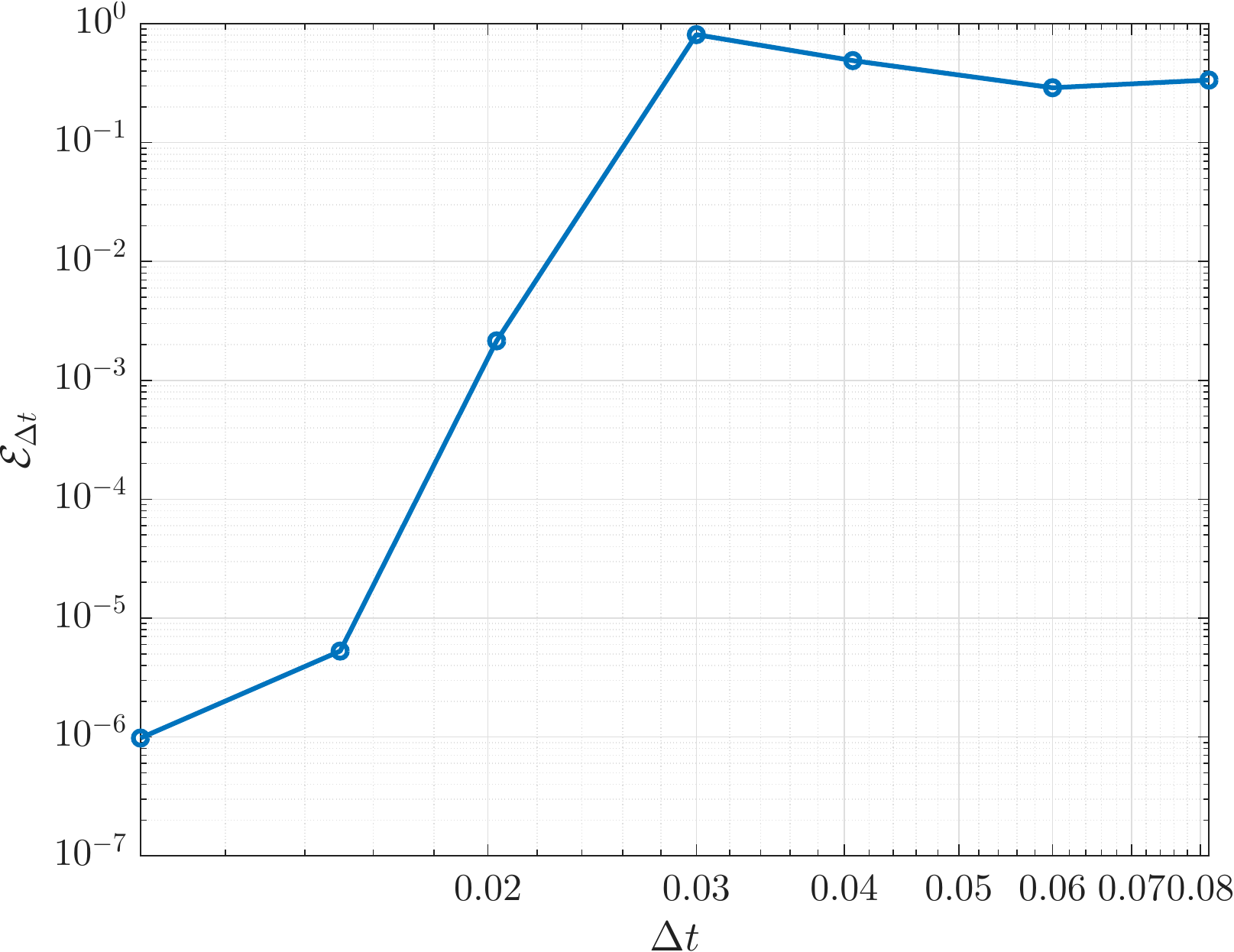}
	\caption{Convergence of the maximum error at $t=6$, estimated on a grid,
          for the eight-source test of Section~\ref{sec:eightSources}.}
	\label{fig:exp6_conv}
\end{figure}

\subsection{Sources and targets on a smooth surface}\label{sec:cruller}

We demonstrate the computational performance of the TK-WFP algorithm compared to a naive direct evaluation in a setting where the source points are also the targets, and they lie on the boundary of a 3D surface. We encounter this evaluation task in time-domain integral-equation schemes for wave scattering problems from 3D obstacles.

We take $M = 102400$ source points with Gaussian pulse time signatures 
\beq
\sigma_j(t) = 10\exp[-\mu_j(t - t_{0,j})^2], \qquad j = 1, \dots, M, 
\eeq
where the peak times are $t_{0,j}= 2 + 5j/M$,
and inverse root widths $\mu_{j}= 30 + 20j/M$,
Thus the peak times lie in $[2,7]$ and the inverse root widths in $[30,50]$.
The largest inverse root width $\mu=50$ is used to estimate
$K_0$ as the frequency at which the Gaussian has dropped from each peak to $\epsilon$,
which gives $K_0 \approx 2 \sqrt{\mu\log(10/\epsilon)} \approx 57$.
We set the Nyquist fraction for the blending $\gamma = 2/3$; the blending region width~\eqref{W} is $W = 14$ time steps. 

We choose source points to be distributed on panels discretizing
a ``cruller'' surface with the parametrization 
\beqs
\begin{split}
\xv(\theta,\psi) &= [(r_2 + H(\theta,\psi)\cos\psi)\cos\theta,\;(r_2 + H(\theta,\psi)\cos\psi)\sin\theta,\;H(\theta,\psi)\sin\psi], \\ 
H(\theta,\psi) &= r_1 + 0.1\cos(5\theta + 3\psi), \quad
\theta\in[0,2\pi], \, \psi \in[0,2\pi],
\end{split}
\eeqs
where $r_1 = 0.3$ and $r_2 = 0.6$.
This surface and its discretization
is a variant of one used as a test case in \cite{Barnett2020},
with different radius parameters chosen to just fit in the cube $B$.

We approximate the solution at the source points themselves using TK-WFP, over $326$ time slices up to a final time $T = 6$, with
the solution time interval $0<t<T$ requiring $N_t = 326$ time steps with $\dt = 0.0184$.
For this $\dt$, the truncation frequency was $K = 171$, needing $N = 313$ frequency modes per dimension at $\dk = 1.0962$.
The typical number of targets within distance $\delta=0.2577$ of each source is
$\Mtyp = 917$.
Recall that when evaluating the potential at the sources themselves,
the self-interaction is excluded.

Comparing against the direct computation of $u$ via \eqref{eq:mainSolnRep2b},
we find the relative maximum error to be $\tilde\calE_{\dt} = 1.8\times10^{-5}$.
Figure~\ref{fig:exp1}(a) shows the computed solution on the cruller at the final time, while the table in Figure~\ref{fig:exp1}(b)
highlights the computational performance of TK-WFP compared to the direct evaluation of $u$ using \eqref{eq:mainSolnRep2b}.
This test is carried out on a single Rome node with two 64-core AMD EPYC 7742 2.25 GHz CPUs (released 2019) and 1000 GB RAM.
One sees from the table that the
TK-WFP method is about 20 times faster than the naive evaluation of $u$.

\begin{figure}[ht]
  (a)\raisebox{-1.85in}{\includegraphics[width=7cm]{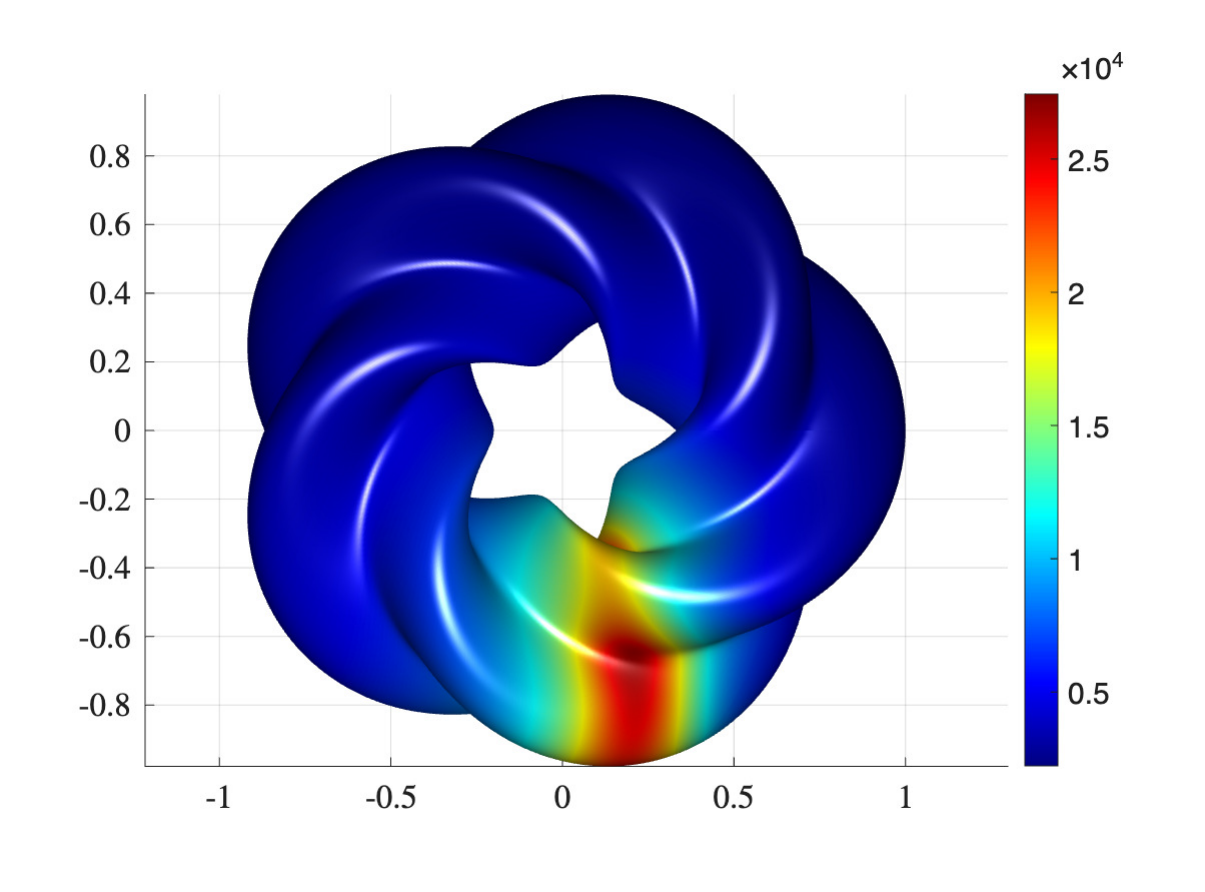}}
\hfill
(b)
\begin{tabular}[t]{|l|l|}
  Task & CPU time \\
  \hhline{|=|=|}
  precomputation                         & 8.19 min   \\ \hline
  $\ul$ eval. per time-step              & 9.57 sec \\
  $\uh$ eval. per time-step              & 2.91 sec \\
  $\alpha(\kv,t)$ update per time-step   & 16.4 sec \\ \hline
  Total per time-step                    & 28.9 sec \\ \hline
  Direct $u$ eval. per time-step         & 10.1 min   \\ \hhline{|=|=|}
  TK-WFP, total for $0\le t\le6$         & 2.76 h  \\ \hline
  Direct eval, total for $0\le t\le 6$   & 54.9 h (est.)  \\ \hline
\end{tabular}
\caption{Results from the surface source and target distribution test
  of Section~\ref{sec:cruller}. Left: computed solution at $t = 6$
  at the $M=102400$ surface points, shown looking down the $z$-axis onto
  the $xy$-plane.
  Color indicates value of $u$, as per the colorscale shown.
  Right: CPU timing table for this experiment. Entries labeled with ``est.'' are estimated
by scaling up from the cost for a single time-step.}
	\label{fig:exp1}
\end{figure}

\subsection{Randomly located sources}\label{sec:randomSources}

We demonstrate the TK-WFP algorithm's ability to handle a large number of sources and targets, in a range of low-to-high frequency settings. Specifically, we test TK-WFP with a million sources and targets, radiating the full range of center frequencies up to $30\pi$
where the computational box $B$ has 30 wavelengths on a side.
Given $M = 10^6$ sources placed at random locations within the box $[-1,1]^3$,
also imposing a $10^{-6}$ minimal distance between any two sources, we choose their time signatures
\beq
\sigma_j(t) = 0.5\left[\erf(5(t - t_{0,j})) + 1\right]\sin(\omega_{j}(t-t_{0,j})), \qquad j = 1, \dots, M.
\eeq
The starting times $t_{0,j}$ and frequencies $\omega_{j}$ are randomly selected
from a set of $M$ equidistant points in $t_{0,j}\in[1.5,5]$ and  $\omega_{j}\in[0,30\pi]$.
Because the maximum frequency and the smooth erf modulation widths are the same as in Section~\ref{sec:eightSources}, the $\epsilon$-bandlimit has the same value $K_0\approx 131$.
At this bandlimit the cube $B$ has a side length about $42$ wavelengths. 
We choose $\gamma = 0.5$ for the Nyquist blending fraction. 

We evaluate the solution up to a final time $T = 6$ over a $100^3$ uniform mesh of target points covering $B$, and $6$ time slices in $[0,T]$.
For the history evolution we use the time step $\dt = 0.012$ and $N_t = 502$.
For the chosen $\dt$, $\delta = 0.2151$, and the typical number of neighboring sources to each target is $\Mtyp = 5214$. The solve requires  a maximum frequency cut-off $K = 263$, $N = 477$ frequency modes, and $\dk = 1.105$. 

Figure~\ref{fig:exp3} presents the computed solution at $t= 4, 6$, while Figure~\ref{fig:exp3_zslice}(a) shows a slice at $t=6$.
The max-norm of the relative error over all space target points at the final time
is $\tilde\calE_{\dt} = 1.4\times 10^{-6}$.

The timing table is shown in Figure~\ref{fig:exp3_zslice}(b), giving a breakdown of the computational performance of TK-WFP compared against the naive direct evaluation of \eqref{eq:mainSolnRep2b}. For the solve over all $\Nt$ steps, we observe an estimated factor of 75 speedup when using TK-WFP over the naive computation.
For this test, a 
Cascade Lake node with four 24-core Xeon Platinum 8268 2.9 GHz CPUs (released 2019), and 2930 GB RAM, was used.
We estimated that the maximum RAM usage of the code was
under 500 GB.

\begin{figure}[th]
	\centering
		\includegraphics[width=7.5cm]{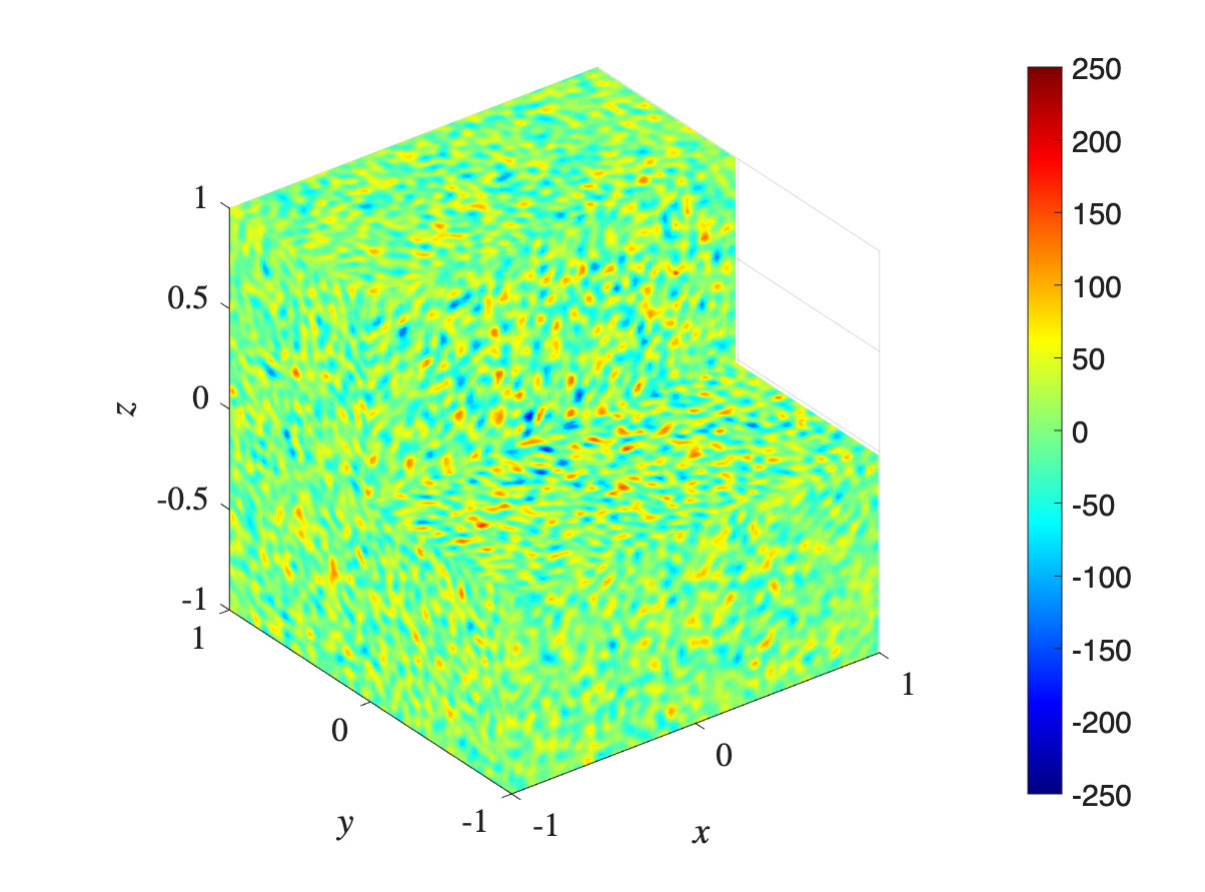}
	\includegraphics[width=7.5cm]{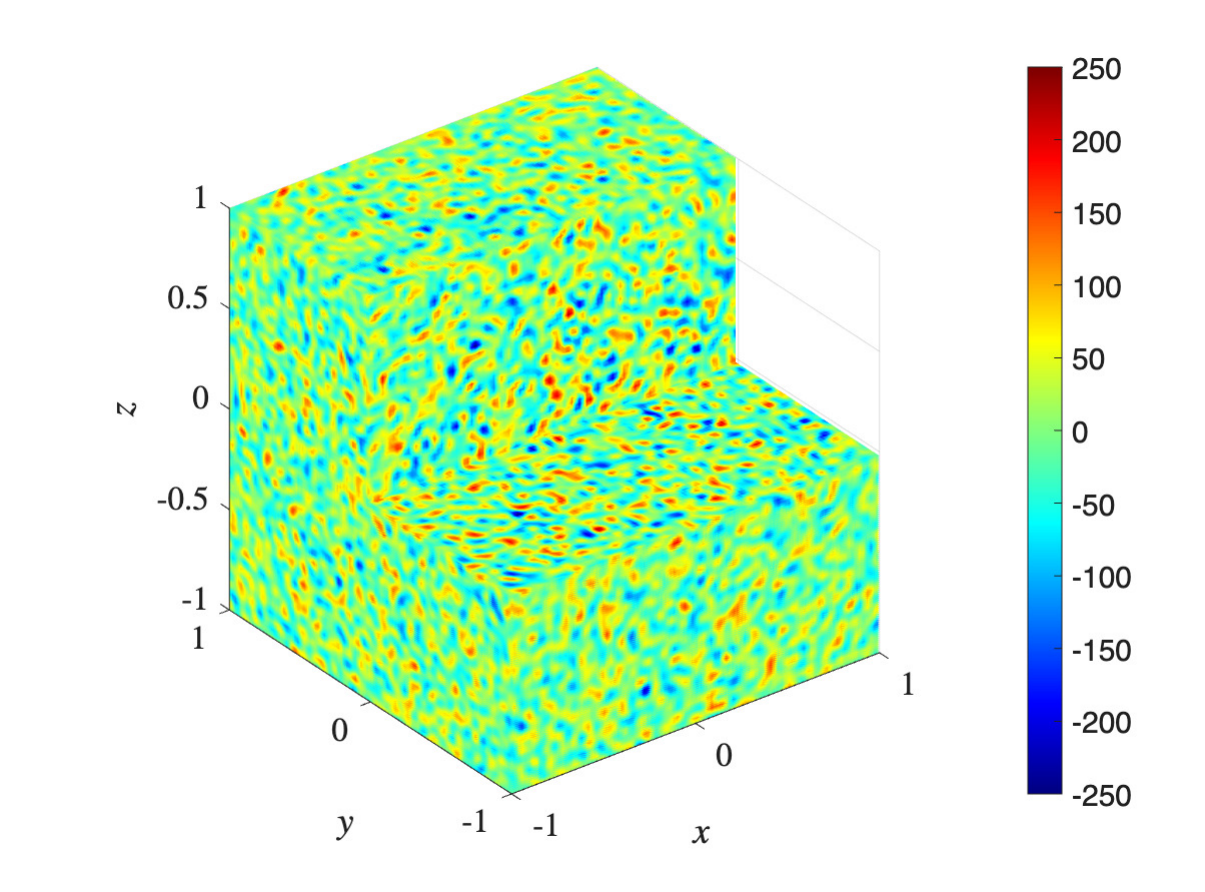}
	\caption{Computed solution at $t = 4$ (left) and $t = 6$ (right), for the test of Section~\ref{sec:randomSources} with $10^6$ sources and targets. The solution is shown on a target grid of size $100^3$. There are $30$ wavelengths per side of cube at the maximum center frequency of the sources, and about 42 wavelengths per side at the $\epsilon$-bandlimit frequency.
          The relative maximum error at $t = 6$ is $\tilde\calE_{\dt} = 1.4\times10^{-6}$.}
	\label{fig:exp3}
\end{figure}

\begin{figure}[th]
(a)
  \raisebox{-1.8in}{\includegraphics[width=7cm]{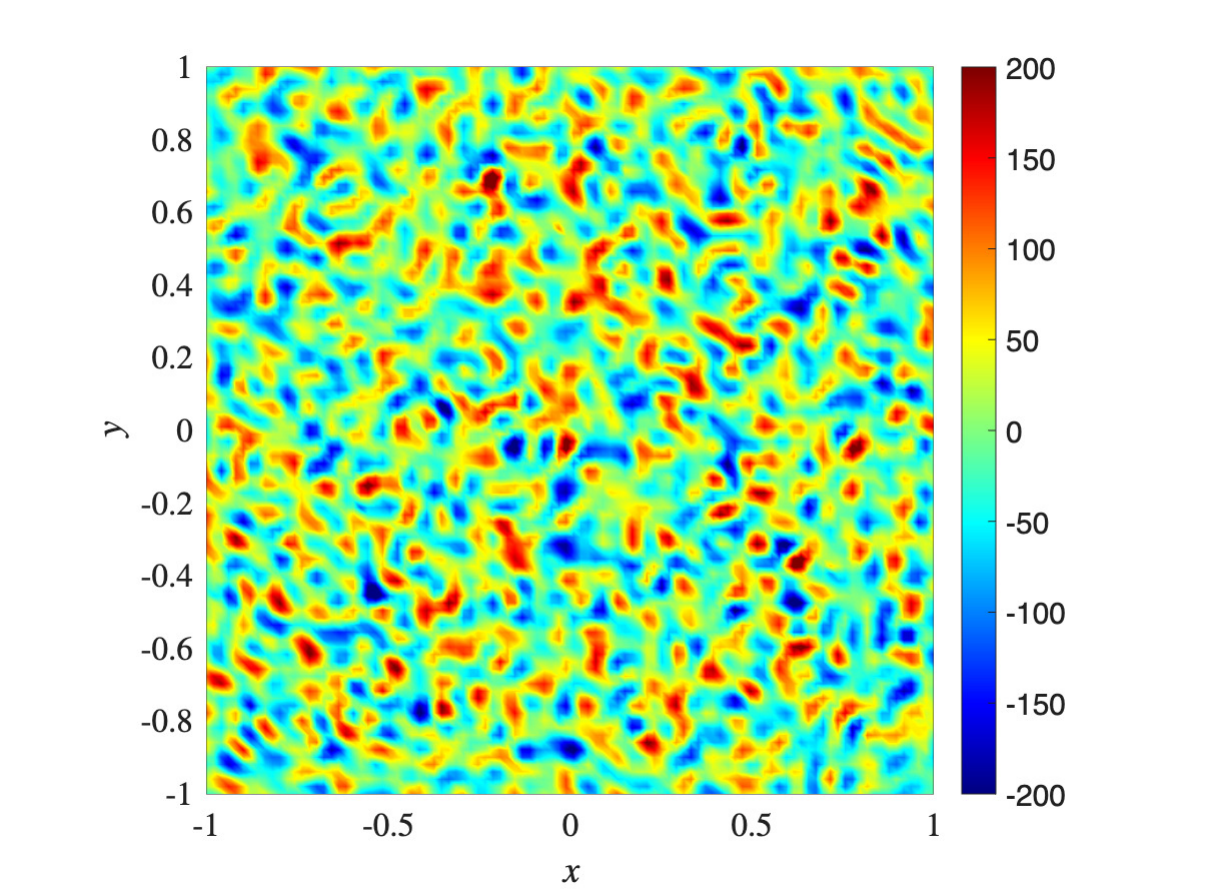}}
\hfill
(b)
\begin{tabular}[t]{|l|l|}
  Task & CPU time \\
  \hhline{|=|=|}
  precomputation                         & 3.36 h   \\ \hline
  $\ul$ eval. per time-step              & 4.23 min \\
  $\uh$ eval. per time-step              & 3.35 sec \\
  $\alpha(\kv,t)$ update per time-step   & 1.58 min \\ \hline
  Total per time-step                    & 5.86 min \\ \hline
  Direct $u$ eval.\ per time-step         & 7.80 h   \\ \hhline{|=|=|}
  TK-WFP, total for $0\le t\le6$         & 52.4 h (est.)   \\ \hline
  Direct eval, total for $0\le t\le 6$   & 3917 h (est.)  \\ \hline
\end{tabular}
\caption{Results from the test of Section~\ref{sec:randomSources} with $10^6$ sources
  and targets.
  Left: computed solution at the final time $t = 6$ on the slice $z = -0.01$.
  Right: CPU timing table for the experiment. Entries labeled with ``est.'' are estimated
by scaling up from the cost for a single time-step.}
\label{fig:exp3_zslice}
\end{figure}

\section{Conclusion and discussion} \label{sec:conclusions}

We have presented the Truncated-Kernel Windowed Fourier Projection (TK-WFP) algorithm for the fast, high-order accurate evaluation of 3D hyperbolic potentials in free space.
The method uses a smooth partition of unity (blending function) to
split the potential into a local part involving only small separations,
plus a smooth remainder (the ``history'' part)
which can be approximated with spectral accuracy by a truncated Fourier integral.
The method limits the history part's rate of oscillation in frequency space via a smooth
spatial truncation of the free-space 3D wave kernel beyond the diameter of the domain;
this exploits strong Huygens and allows error-free Fourier discretization.
The full history dependence
is compressed by exploiting a Duhamel-type update formula for each Fourier coefficient;
the driving terms in this formula only involve the two narrow time intervals where the partition
of unity is changing (creation and annihilation, as in Remark~\ref{r:crea}).
This reduces the cost of evaluation from one that is quadratic with respect to the number of spatial and temporal points to one whose cost involves a linear term in the number of sources, plus a quasilinear term in the number of Fourier modes.
Although the latter term does not result in the optimal scaling in the case of surface
distributions, the scheme has a small prefactor and, being a two-level scheme,
is much simpler to implement than hierarchical time-domain FMMs.
(This is also true of the two-level FFT scheme of \cite{Yilmaz2004}.)

We included a theoretical result that the Fourier truncation error drops exponentially
with respect to the additional bandlimit introduced by the blending, when a Kaiser--Bessel
function is used.
We demonstrated the accuracy and efficiency of the scheme through three numerical examples, including one with a million sources and targets spanning more than 30 wavelengths per side of the computational domain.

The TK-WFP method can serve as an acceleration tool for the evaluation step in a variety of
time-domain 3D wave-scattering scenarios:
a large number of point-like (s-wave) scatterers, line-like scatterers
(eg conductors using a slender-body approximation),
surface boundary conditions, or a volumetric variable medium.
In particular, our future plans include coupling the scheme
with an integral representation, high-order surface quadratures, and a Volterra time
marching scheme, to accelerate the method of \cite{Barnett2020}.

\begin{remark}[Periodic case and FFT size ratio] \label{r:peri}  
  If a periodic boundary condition is needed
  instead of the free space condition presented in this paper, two changes are
  needed: i) using $G$ instead of the truncated kernel $\CGA$, and
  ii) fixing $\dk = \pi$ to give a Fourier series on $[-1,1]^3$.
  We see that the ratio of \eqref{dkprac} to this $\dk$ is $A/2+1 \approx \sqrt{3}+1$
  (when $\delta\ll1$), meaning that the NUFFT coefficient arrays (and hence FFTs)
  are $(\sqrt{3}+1)^3 \approx 20$ times larger for TK-WFP than periodic WFP, all other
  things being equal.
  This ratio does not directly impact the method cost, because of rebalancing
  of local vs history parts. However, it does motivate further work on
  spectrally-accurate radiative conditions.
  Note that for {\em interior} wave problems (e.g.\, room acoustics), periodic conditions
  on an enclosing box are sufficient.
\end{remark}
  
\begin{remark}[Connection to prolate-based Ewald methods]   
  The 3D WFP split of \eqref{ul}--\eqref{uh} may be viewed as a hyperbolic
  version of the Ewald split for 3D elliptic PDE (Poisson equation for periodic
  electrostatics);
  the latter traditionally uses Gaussians or, recently, prolates (see \cite{ESP25}
  and references within).
  The Kaiser--Bessel window $\phi'$ in WFP is a very close approximation
  to the prolate function used in \cite{ESP25},
  and in both methods the antiderivative is taken to get the blending function.
  However, in \cite{ESP25} only one half of the domain of the antiderivative is used,
  making the cutoff radius (analogous to $\delta$) half the size of $\delta$ used
  in WFP for the same spectral bandlimit.
  It would be interesting to explore if i) WFP could be modified in the same way
  to reduce the cost of the local evaluation by a factor of 8, and
  ii) if the upsampling, currently implicit in the NUFFTs for WFP,
  can be avoided as in the 5-step particle mesh Ewald procedure.
\end{remark}

\section*{Acknowledgements}
We are grateful for discussions with Tom Hagstrom.
The Flatiron Institute is a division of the Simons Foundation.

\appendix

\section{Proof of Lemma~\ref{thm:windowFourierLem} on the decay of the Fourier transform of the window function}\label{sec:windowFourierProof}
Starting from the Fourier transform of the Kaiser-Bessel window derivative $\phi'$ in~\eqref{eq:windowFT}, we have
\beq
\abs{\hat{\phi'}(\omega)}\leq \frac{b}{\sinh b} \frac{1}{\sqrt{\left(\frac{\delta\omega}{2}\right)^2 - b^2}}, \quad |\omega| \ge \frac{2b}{\delta},
\eeq
since the sinc is bounded by 1.
Noting that 
$
\sinh b = (1 - \epsilon^2)/2\epsilon < 1/2\epsilon, 
$
for $0<\epsilon<1$, we get
\beq
\abs{\hat{\phi'}(\omega)}< \frac{4\epsilon b}{\delta\abs{\omega}}\frac{1}{\sqrt{1 - \left(\frac{2b}{\delta \omega}\right)^2}},
 \quad |\omega| \ge \frac{2b}{\delta}.
\eeq
For $\theta>1$, rearranging the inequality $\abs{\omega}\geq 2b/\delta \sqrt{1 - 1/\theta^2}$ gives
$
1/\sqrt{1 - \left(\frac{2b}{\delta\omega}\right)^2}\leq \theta, 
$
and~\eqref{eq:KB_bandlimit} follows.~\qed

\section{Proof of Theorem~\ref{thm:alphaDecay} on the decay of Fourier coefficients}\label{sec:alphaDecayProof}
We first establish the following Lemma on the Fourier transform of the pointwise product of two general functions whose Fourier transform satisfies decay properties for large frequencies. 

\renewcommand{\appendixname}{} 
\renewcommand{\thesection}{\Alph{section}} 

\begin{lem}\label{thm:FTofProduct}
Let $f,g \in L^2(\mathbb R)$, $h(t) = f(t)g(t)$, $t\in\mathbb R$, and $0<\epsilon< 1$. Let $\hat{f}$, $\hat{g}$, and $\hat{h}$ be the Fourier transforms of $f,g$ and $h$ respectively, as per the definition
\beqs
\hat{f}(\omega) = \int_{\mathbb R} e^{i\omega t} f(t) dt.
\eeqs
Suppose there exists $C_1$ and $K_1$, and $C_2$ and $K_2$, such that 
\beq
\abs{\hat{f}(\omega)}\leq\frac{C_1\epsilon}{\omega^2}, \quad \text{for } \abs{\omega}\geq K_1, \quad \text{and} \quad \abs{\hat{g}(\omega)}\leq\frac{C_2\epsilon}{\omega^2}, \quad \text{for } \abs{\omega}\geq K_2.
\eeq
Then 
 \beq
 \abs{\hat{h}(\omega)} < \tilde C\frac{\epsilon}{\omega^2}, \quad \text{ for }\abs{\omega}\geq K_1 + K_2, 
 \eeq
 where 
 $$
 \tilde C = \frac{1}{2\pi}\left[C_2\norm{\hat{f}}_1\left(1 + \frac{K_1}{K_2}\right)^2
 + C_1\norm{\hat g}_1\left(1 + \frac{K_2}{K_1}\right)^2
 + 8C_1C_2\left(\frac{1}{K_1} + \frac{1}{K_2}\right)
  \right]. 
$$
\end{lem}
\begin{proof}
By the convolution theorem 
\beq
\hat{h}(\omega) = \frac{1}{2\pi} (\hat{f}\ast\hat{g})(\omega) = \frac{1}{2\pi} \int_{\mathbb R} \hat{f}(\xi) \hat{g}(\omega - \xi)d\xi. 
\eeq
Fix $\abs{\omega}\geq K_1 + K_2$, and consider partitioning $\hat{h}(\omega)$ into 
\beq
\hat{h}(\omega) = \frac{1}{2\pi} \bigl(H_1(\omega) + H_2(\omega) + H_3(\omega) \bigr),
\eeq 
where 
\beq
H_l(\omega) : = \int_{R_l(\xi)} \hat{f}(\xi)\hat{g}(\omega - \xi)d\xi,\quad l = 1, 2, 3,  
\eeq
and
\beq
\begin{split}
&R_1(\xi):=  \left\{\xi \in\mathbb R: \abs{\xi}\leq K_1\right\}, \\
&R_2(\xi):=  \left\{\xi \in\mathbb R: \abs{\xi}>K_1, \abs{\omega - \xi}\geq K_2\right\}, \\
&R_3(\xi):=  \left\{\xi \in\mathbb R: \abs{\xi}>K_1, \abs{\omega - \xi}<K_2\right\}. 
\end{split}
\eeq
We seek upper bounds on each of $H_l$, $l = 1,2, 3$. 

In $R_1(\xi)$, and for $\abs{\omega}\geq K_1 + K_2$, $\abs{\omega - \xi}\geq\abs{\omega}- \abs{\xi}\geq\abs{\omega} - K_1\geq K_2$, then 
\beq
\abs{H_1(\omega)} \leq \int_{R_1(\xi)} \abs{\hat{f}(\xi)}\abs{\hat{g}(\omega - \xi)}d\xi\leq \int_{R_1(\xi)}\frac{C_2\epsilon\abs{\hat{f}(\xi)}}{\abs{\omega - \xi}^2}d\xi\leq \frac{C_2\epsilon}{(\abs{\omega} - K_1)^2}\norm{\hat{f}}_1. 
\eeq
Similarly, using a change of variables, we have 
\beq
\abs{H_3(\omega)} \leq \int_{R_3(\omega - \eta)} \abs{\hat{f}(\omega - \eta)}\abs{\hat{g}(\eta)}d\eta\leq \frac{C_1\epsilon}{(\abs{\omega} - K_2)^2}\norm{\hat{g}}_1. 
\eeq
Finally consider 
\beq
\abs{H_2(\omega)} \leq \int_{R_2(\xi)}\abs{\hat{f}(\xi)}\abs{\hat{g}(\omega - \xi)}d\xi\leq C_1C_2\epsilon^2\int_{R_2(\xi)}\frac{1}{\xi^2(\omega - \xi)^2}d\xi,
\eeq
and split $R_2$ further into 
\beq
\begin{split}
&R_{2,1}(\xi):=\left\{\xi \in\mathbb R: \abs{\xi}>K_1, \abs{\omega - \xi}>K_2, \abs{\xi}\leq \abs{\omega}/2\right\}, \\
&R_{2,2}(\xi):= \left\{\xi \in\mathbb R:\abs{\xi}>K_1, \abs{\omega - \xi}>K_2, \abs{\xi}> \abs{\omega}/2\right\}. \\
\end{split}
\eeq
In $R_{2,1}$, $\abs{\omega - \xi}\geq \abs{\omega} - \abs{\xi} \geq \abs{\omega}/2$ so that 
\beq
\int_{R_{2,1}(\xi)} \frac{1}{\xi^2(\omega - \xi)^2}d\xi\leq \frac{4}{\omega^2}\int_{R_{2,1}(\xi)}\frac{1}{\xi^2}d\xi \leq \frac{8}{K_1\omega^2}. 
\eeq
Similarly in $R_{2,2}$, a change of variables gives 
\beq
\int_{R_{2,2}(\xi)}\frac{1}{\xi^2(\omega - \xi)^2}d\xi = \int_{R_{2,2}(\omega - \eta)}\frac{1}{\eta^2(\omega - \eta)^2}d\eta\leq \frac{8}{K_2\omega^2}. 
\eeq
Hence, 
\beq
\abs{H_2(\omega)} \leq \frac{8C_1C_2\epsilon^2}{\omega^2}\left(\frac{1}{K_1} + \frac{1}{K_2}\right). 
\eeq
Therefore, 
 \beq
 \abs{\hat{h}(\omega)}\leq \frac{1}{2\pi} \left[\frac{C_2\norm{\hat f}_1}{(\abs{\omega} - K_1)^2}\epsilon+ \frac{8C_1C_2}{\omega^2}\left(\frac{1}{K_1}+ \frac{1}{K_2}\right)\epsilon^2 + \frac{C_1\norm{\hat g}_1}{(\abs{\omega} - K_2)^2}\epsilon\right].
 \eeq
 Note that $\abs{\omega} - K_1> \frac{K_2}{K_1 + K_2}\abs{\omega}$, whenever $\abs{\omega} > K_1 + K_2$, and similarly $\abs{\omega} - K_2> \frac{K_1}{K_1 + K_2}\abs{\omega}$, then 
 \beq
 \abs{\hat{h}(\omega)}\leq \frac{1}{2\pi}\frac{1}{\omega^2} \left[C_2\norm{\hat{f}}_1\left(1 + \frac{K_1}{K_2}\right)^2\epsilon
 + C_1\norm{\hat g}_1\left(1 + \frac{K_2}{K_1}\right)^2\epsilon
 + 8C_1C_2\left(\frac{1}{K_1} + \frac{1}{K_2}\right)\epsilon^2
  \right]. 
 \eeq
 Noting that $\epsilon^2<\epsilon$ completes the proof.
\qed
\end{proof}
\begin{remark}
If $f\in L^{2}(\mathbb R)$ such that 
$
\abs{\hat{f}(\omega)} \leq C_1\epsilon\omega^{-2}, \ \text{for } \abs{\omega} \geq K_1, 
$
then $\norm{\hat{f}}_1<\infty$ because 
\beq\label{eq:normOfFT}
\norm{\hat{f}}_1 = \int_{\abs{\omega}<K_1}\abs{\hat{f}(\omega)} d\omega + \int_{\abs{\omega}>K_1}\abs{\hat{f}(\omega)}d\omega
\leq \sqrt{2K_1}\norm{\hat{f}}_2 + \frac{2C_1\epsilon}{K_1}<\infty,
\eeq
using Cauchy-Schwarz inequality and the decay of $\hat{f}$. 
\end{remark}

We now prove Theorem~\ref{thm:alphaDecay}. The vanishing properties of $\phi$ allow the extension of the integration limits in the definition of $\alpha(\kv,t)$ in~\eqref{eq:alphak} to the real line, so that
\beq
\alpha(\kv,t) = \int_\mathbb{R} \frac{\sin\kappa(t - \tau)}{\kappa}\sighat(\kv,\tau) \phi(t - \tau)\phi(A - t + \tau)d\tau.
\eeq
Define 
$$\chi(\kv,t;\tau) := \sighat(\kv,\tau) \psi(t-\tau),$$ 
where 
\beq
\psi(\tau) :=
\phi(\tau)\phi(A - \tau) = 
\phi(\tau)[1-\phi(\tau - A + \delta)] = \phi(\tau)-\phi(\tau - A + \delta),
\eeq
where we exploited the inversion symmetry $\phi(\delta-\tau)=1-\phi(\tau)$ and
other properties of $\phi$.
We find that for $t\in[0,T]$, 
\beq
\psi(\tau) = \int_{\tau - A + \delta}^{\tau} \phi'(s)ds, \text{ and }
\hat\psi(\omega) = \frac{\hat{\phi'}(\omega)}{i\omega}\left(e^{i\omega(A - \delta)}- 1\right). 
\eeq
Therefore, by Lemma~\ref{thm:windowFourierLem}, and for some $\theta>1$, we have
\beq
\abs{\hat{\psi}(\omega)} \leq \frac{2\abs{\hat{\phi'}(\omega)}}{\abs{\omega}}< \frac{4b\theta}{\delta}\cdot\frac{\epsilon}{\omega^2}, \quad \text{whenever } \abs{\omega}\geq \frac{2b}{\delta\sqrt{1 - \theta^{-2}}}. 
\eeq
Furthermore, for fixed $\kv$ and $t$, the function $\chi(\kv,t;\tau)$ is $L^2(\mathbb R)$ with respect to $\tau$, and is a product of $S(\kv,t)$ and $\psi(t-\tau)$, whose Fourier transforms with respect to $\tau$ satisfy the following decay properties
\beq
\abs{\hat{S}(\kv,\omega)} \leq M\cdot\frac{\epsilon}{\omega^2}, \quad \abs{\omega}\geq K_0, \quad \text{and} \quad
\abs{\hat{\psi}(\omega)} <\frac{4b\theta}{\delta}\cdot\frac{\epsilon}{\omega^2}, \quad \abs{\omega}\geq K_1, 
\eeq
where $K_1:= 2b/\delta\sqrt{1 - \theta^{-2}}$. Then by Lemma~\ref{thm:FTofProduct}, the Fourier transform of $\chi$ with respect to $\tau$, denoted by $\hat{\chi}(\kv,t;\omega)$, satisfies the property
\beq\label{eq:chiProperty}
\abs{\hat \chi(\kv,t;\omega)} < c_1\frac{\epsilon}{\omega^2}, \quad \text{for } \abs{\omega}>K, 
\eeq
where $K = K_0+K_1$ is the total bandlimit as stated in \eqref{albnd} in the theorem, and
$$
c_1 = \frac{1}{2\pi} \left[\frac{4b\theta}{\delta}\left(\sum_{j = 1}^M\norm{\hat \sigma_j}_1\right)\left(1 + \frac{K_0}{K_1}\right)^2 
+ M\norm{\hat\psi}_1\left(1 + \frac{K_1}{K_0}\right)^2
+ 8\cdot \frac{4b\theta}{\delta} M\left(\frac{1}{K_0} + \frac{1}{K_1}\right)\right]. 
$$
With the following upper bounds on $\norm{\hat{\sigma}_j}_1$ and $\norm{\hat{\psi}}_1$,
\beq\label{eq:oneNormBounds}
\norm{\hat{\sigma}_j}_1\leq 2K_0\norm{\sigma_j}_1 + \frac{2\epsilon}{K_0}, \quad \norm{\hat{\psi}}_1\leq 2K_1A + \frac{4b\theta}{\delta}\cdot\frac{2\epsilon}{K_1}, 
\eeq
and the property that $\norm{\sigma_j}_1\leq P$ for all $j = 1, \dots, M$, we have
\beq\label{eq:c1inequality}
c_1 \leq \frac{1}{2\pi} \left[\frac{8b\theta M}{\delta}\left(K_0P + \frac{1}{K_0}\right)\left(\frac{K}{K_1}\right)^2 
+ 2M\left(K_1 A + \frac{4b\theta}{\delta K_1}\right)\left(\frac{K}{K_0}\right)^2 + 8\cdot\frac{4b\theta M}{\delta}\left(\frac{K}{K_0K_1}\right)\right]. 
\eeq
Plugging the expression of $K_1$, noting that $\sqrt{1 - \theta^{-2}}<1$, and simplifying gives
\beq\label{eq:Cdef}
c_1< CM, \quad
\mbox{ where } C := \frac{K}{\pi K_0\sqrt{1 - \theta^{-2}}}\left\{\frac{2AbK}{\delta K_0} + \left(\theta - \theta^{-1}\right)\left[\frac{\delta K}{b}(K_0^2P + 1) + \frac{2K}{K_0} + 8\right]\right\}. 
\eeq
We use Equation~\eqref{eq:chiProperty} and the inequality in \eqref{eq:Cdef} to establish the decay property on $\alpha(\kv,t)$ given in~\eqref{albnd}. 
Writing $\sin\kappa(t - \tau)$ in terms of complex exponentials gives
\beq
\begin{split}t
\alpha(\kv,t) &= \frac{1}{2i\kappa}\left[e^{i\kappa t}\int_\mathbb{R} e^{-i\kappa\tau}\chi(\kv,t;\tau)d\tau - e^{-i\kappa t}\int_\mathbb{R} e^{i\kappa\tau}\chi(\kv,t;\tau)d\tau\right], \\
& = \frac{1}{2i\kappa}\left[e^{i\kappa t}\hat\chi(\kv,t;-\kappa) - e^{-i\kappa t}\hat\chi(\kv,t;\kappa)\right],
\end{split}
\eeq
so that 
\beq
\abs{\alpha(\kv,t)} \leq \frac{1}{2\kappa} \left(\abs{\hat{\chi}(\kv,t;-\kappa)} + \abs{\hat{\chi}(\kv,t;\kappa)}\right). 
\eeq
Finally, applying Equation~\eqref{eq:chiProperty} and \eqref{eq:c1inequality} yields
the claimed inequality \eqref{albnd},
with the constant $C$ as in \eqref{eq:Cdef}.
As claimed, \eqref{eq:Cdef} shows that $C$ grows weakly with respect to the
parameters (at most quadratic with respect to bandlimit, when considering combined powers of $K_0$, $K$, and $b/\delta$, and linear in all other parameters).
\qed

\bibliographystyle{elsarticle-num}
\bibliography{refs}
 
\end{document}